\documentclass[12pt,reqno]{amsart}
\usepackage{amsfonts,amsmath,amssymb,amsthm}
\usepackage{latexsym}
\usepackage{eucal}
\usepackage[ansinew]{inputenc}
\usepackage[american]{babel}
\usepackage{amsfonts}
\usepackage{amssymb}
\usepackage[dvips]{graphicx}

\numberwithin{equation}{section}

\newtheorem{teo}{Theorem}[section]
\newtheorem{lema}{Lemma}[section]
\newtheorem{prop}{Proposition}[section]

\theoremstyle{remark}
\newtheorem{remark}{Remark}[section]
\newtheorem{obs}{Remark}[section]

\newcommand{\R}{\mathbb{R}}

\begin{document}
%\setpagewiselinenumbers
%\modulolinenumbers[1]
%\line-numbers

\title[Boussinesq-Full Dispersion]
{Existence of solitary waves solutions for internal waves in two-layers systems}

\author[J. Angulo and J-C Saut]{Jaime Angulo Pava$^1$ and Jean-Claude Saut $^2$}
%\address[angulo@ime.usp.br]{{\sc Jaime Angulo Pava}: Institute of Mathematics and Statistics, State
%University of S\~ao Paulo, S\~ao Paulo, SP, Brazil.}
%\address[Jean-Claude.Saut@math.u-psud.fr]{{\sc Jean-Claude Saut}: Laboratoire de Math\'ematiques, UMR 8628, Universit\'e Paris-Sud et CNRS, 91405 Orsay, France.}

\maketitle

\centerline{$^1$Department of Mathematics,
IME-USP}
 \centerline{Rua do Mat\~ao 1010, Cidade Universit\'aria, CEP 05508-090,
 S\~ao Paulo, SP, Brazil.}
 \centerline{\it angulo@ime.usp.br}

 \centerline{$^2$Laboratoire de Math\'ematiques d'Orsay }
 \centerline{ Univ. Paris-Sud, CNRS, Universit\' e Paris-Saclay, 91405 Orsay, France}
 \centerline{\it Jean-Claude.Saut@u-psud.fr}
%%%%%%%%%%%%%%%%%%%%%%%%%%%%%%%%%%%%%%%%%%%%
\begin{abstract}
The aim  of this paper is to establish the existence of solitary wave solutions for two classes of  two-layers systems modeling  the propagation of internal waves. More precisely we will consider  the  Boussinesq-Full dispersion system and the Intermediate Long Wave (ILW) system together with its Benjamin-Ono (BO) limit.
\end{abstract}

%\maketitle


\thanks{{\it Date}: April 3, 2018}

%%%%%%%%%%%%%%%%%%%%%%%%%%%%%%%%%%%%%%%%%%%%%%%%%%
%  INTRODUCTION
%%%%%%%%%%%%%%%%%%%%%%%%%%%%%%%%%%%%%%%%%%%%%%%%%%

\section{Introduction}
This paper is concerned with the existence and properties of solitary wave solutions to some internal waves models. Such models were systematically and  rigorously (in the sense of consistency) derived from the two-layers system with a rigid lid in  \cite{BLS} (see also \cite{CGK} and the survey article \cite{Sa}). They depend of course on the various regimes determined by the wave lengths, heights of the layers,.. that determine the suitable small parameters used in the derivation of the  asymptotic models.

For instance, it is shown  \cite{BLS} that in the so-called Boussinesq-Full dispersion regime  and in the absence of surface tension, the two-layers system for internal waves is consistent with the {\it three-parameter family} of Boussinesq-Full dispersion systems (henceforth B-FD systems)
\begin{equation}\label{BF}
\left\{\begin{array}{lll}
(1-\mu b\Delta)\partial_t\zeta + \frac1\gamma \nabla \cdot((1-\epsilon\zeta){\bf{v}}_\beta)\\
\;\;\;\;- \frac{\sqrt{\mu}}{\gamma^2} |D| \coth(\sqrt{\mu_2}|D|)\nabla\cdot {\bf{v}}_\beta + \frac{\mu}{\gamma}\big(a-\frac{1}{\gamma^2} \coth^2(\sqrt{\mu_2}|D|)\Big)\Delta\nabla\cdot {\bf{v}}_\beta=0,\\
(1-\mu d\Delta)\partial_t {\bf{v}}_\beta+ (1-\gamma) \nabla \zeta - \frac{\epsilon}{2\gamma}\nabla|{\bf{v}}_\beta|^2 + \mu c(1-\gamma)\Delta\nabla \zeta=0,
\end{array}\right.
\end{equation}
where ${\bf{v}}_\beta=(1-\mu\beta \Delta)^{-1}{\bf{v}}$ (${\bf v} $ being the horizontal velocity) and the constants $a, b, c$ and $d$ are defined as
$$
a=\frac13(1-\alpha_1-3\beta),\quad b=\frac13\alpha_1,\quad c=\beta\alpha_2,\quad d=\beta(1-\alpha_2),
$$
with $\alpha_1\geqq 0$, $\beta\geqq 0$, and $\alpha_2\leqq 1$. Note that $a+b+c+d=\frac{1}{3}.$

To describe precisely the asymptotic regime we are considering we recall the definition of a few parameters (see \cite{BLS} for more details). The upper (resp. lower) layer is indexed by $1$ (resp. $2$.) $\gamma =\frac{\rho_1}{\rho_2}$ is the ratio of densities. $d=\frac{d_1}{d_2}$ is the ratio of the typical depth of the two layers. The typical elevation of the wave is denoted by a, the typical horizontal wavelength is $\lambda$. Then, one set

$$\epsilon=\frac{a}{d_1},\; \mu=\frac{d_1^2}{\lambda ^2},\; \epsilon_2=\frac{a}{d_2}=\epsilon \delta,\; \mu_2=\frac{d_2^2}{\lambda^2}=\frac{\mu}{\delta^2}.$$

The Boussinesq-Full dispersion regime is obtained by assuming that $\mu\sim \epsilon$ and $\mu_2\sim 1$ (and thus $\delta\sim \epsilon^{1/2}$) in addition to $\epsilon\ll 1$ and $\epsilon_2\ll 1.$ The three free parameters in \eqref{BF} arise from a double use of the so-called BBM (Benjamin-Bona-Mahony)  trick and of a one-parameter choice of the velocity (see \cite{BLS} for details). 

\begin{remark}
If in addition to $\epsilon \ll 1$ and $\epsilon_2\ll1$ one assumes that $\mu\sim\epsilon$ and $\mu_2\sim \epsilon_2$ and thus $\delta\sim 1,$ one obtains the {\it Boussinesq-Boussinesq} regime, leading to Boussinesq systems similar to those introduced in \cite{BCS1, BCS2, BCL} and studied in \cite {SX, Bu, SWX}.
\end{remark}

\begin{remark}
In oceanographic applications the surface tension effects are very weak and can be ignored when deriving the aforementioned asymptotic models (they appear as a lower order effect). In situations where they are small but of the order of the "small" parameters involved in the asymptotic expansions (the $\epsilon's$ or the $\mu's$) one has to add a cubic "capillary" term to the equation for ${\bf v},$ proportional to $\Delta\nabla \zeta, $  but also add higher oder terms in $\epsilon, \mu$ leading to a somewhat more  complicated system (see \cite{Cung2}) that we will not consider here.
\end{remark}

It is easily checked that (\ref{BF}) is linearly well posed when (see \cite{Cung})
$$
a\leqq 0, \;c\leqq 0,\; b\geqq 0,\; d\geqq 0.
$$
Throughout the paper we will assume that this condition holds.

The local well-posedness of the Cauchy problem for (\ref{BF}) was considered by Cung in \cite{Cung} in the following cases
\begin{enumerate}
\item[(1)] $b>0$, $d>0$, $a\leqq 0$, $c<0$;
\item[(2)] $b>0$, $d>0$, $a\leqq 0$, $c=0$;
\item[(3)] $b=0$, $d>0$, $a\leqq 0$, $c=0$;
\item[(4)] $b=0$, $d>0$, $a\leqq 0$, $c<0$;
\item[(5)] $b>0$, $d=0$, $a\leqq 0$, $c=0$.
\end{enumerate}

Note that all the linearly well-posed systems are not covered here, in particular those where both the coefficients $b$ and $d$ vanish (but never happens in the modeling of internal waves without surface tension due to the modeling constraint $a+b+c+d=\frac{1}{3}.$ Note also that long time existence (that is on time scales of order $\frac{1}{\epsilon}\sim \frac{1}{\mu})$ is still an open problem and will be considered elsewhere.

On the other hand \eqref{BF} is an Hamiltonian system when $b=d.$ The Hamiltonian structure is given by
\begin{equation}
\partial_t \begin{pmatrix}\zeta\\{\bf v}_\beta \end{pmatrix}+J\text{grad}\;H(\zeta, {\bf v}_\beta)=0,
\end{equation}
where
$$J=(1-\mu b\Delta)^{-1}\begin{pmatrix}0&\nabla\cdot\\\nabla  &0\end{pmatrix}$$
and 
%\begin{paragraph}
\begin{multline}
 H(\zeta, {\bf v}_\beta)=\int_{\mathbb R} (\frac{1-\gamma}{2}\zeta^2+\frac{1}{2}|{\bf v}_\beta|^2-\frac{\epsilon}{2\gamma}\zeta|{\bf v}_\beta|^2-\frac{\mu c}{2} (1-\gamma)|\nabla \zeta|^2\\-\frac{a\mu}{2\gamma}|\nabla {\bf v}_\beta|^2+\frac{\sqrt \mu}{2\gamma^2}|L_1^{1/2}{\bf v}_\beta|^2+\frac{\mu}{2\gamma^3}|L_2^{1/2} {\bf v}_\beta|^2), 
 \end{multline}
% \end{paragraph}
where $L_1=|D|\coth(\sqrt \mu|D|)$  and $L_2=\coth(\sqrt \mu|D|).$

We are interested here in the existence of solitary waves solutions for the B-FD  systems in one space dimension. Denoting now $v$  the velocity, the system writes :
 \begin{equation}\label{BF1}
\left\{\begin{array}{lll}
J_b\partial_t\zeta +\mathcal L_{\mu_2}\partial_x v -\frac{\epsilon}{\gamma}\partial_x(\zeta v)=0\\
J_d\partial_t v+ (1-\gamma) J_c\partial_x \zeta -\frac{\epsilon}{2\gamma}\partial_x(v^2)=0,
\end{array}\right.
\end{equation}
where $\mathcal L_{\mu_2}$ is the self-adjoint operator defined for $\mu_2\in (0, +\infty]$ by
\begin{equation}\label{L2}
\mathcal L_{\mu_2}=\frac1\gamma-\frac{\sqrt{\mu}}{\gamma^2} |D| \coth(\sqrt{\mu_2}|D|)+ \frac{\mu}{\gamma}\big(a-\frac{1}{\gamma^2} \coth^2(\sqrt{\mu_2}|D|)\Big)\partial_x^2,
\end{equation}
We note that for $\mu_2=+\infty$ the linear operator $\mathcal L_{\mu_2}\equiv \mathcal L_{+\infty}$ becomes
\begin{equation}\label{Linf}
\mathcal L_{+\infty}=\frac1\gamma-\frac{\sqrt{\mu}}{\gamma^2} |D| + \frac{\mu}{\gamma}\big(a-\frac{1}{\gamma^2} \Big)\partial_x^2.
\end{equation}
The Bessel-type operators $J_b, J_d, J_c$ are defined by
$$
J_b=1-\mu b\partial_x^2,\; J_d=1-\mu d\partial_x^2, \; J_c=1+\mu c\partial_x^2.
$$
 The solitary waves solutions $\zeta(x,t)=\xi(x-\omega t), v(x,t)=\nu(x-\omega t)$ vanishing at infinity should satisfy in the variable $z=x-\omega t$, $\xi=\xi(z), \nu=\nu(z)$ the system
\begin{equation}\label{0soli}
 \left\{\begin{array}{lll}
-\omega J_b \xi +\mathcal L_{\mu_2}\nu =\frac{\epsilon}{\gamma}\xi \nu\\
-\omega J_d\nu+ (1-\gamma) J_c\xi =\frac{\epsilon}{2\gamma}\nu^2.
\end{array}\right.
\end{equation}

The following theorem establishes our results on the existence of solitary waves solutions for systems (\ref{BF1}) in the Hamiltonian case $b=d$. More exactly, we assume
\begin{equation}\label{abc}
a+c+2b=\frac13, \qquad\quad a,c\leqq 0,\;b\geqq 0, 
\end{equation}
implying thus  $b\geqq \frac16$ and $\frac13 -2b\leqq a\leqq 0$ and $\frac13 -2b\leqq c\leqq 0$.
\begin{teo}\label{main}
Let $\gamma \in (0,1)$ and $\epsilon>0$. For $b=d>0$ and $a, c\leqq 0$, the systems (\ref{soli}) have a non-trivial solitary waves solutions $(\xi, \nu)\in H^\infty(\mathbb R)\times H^\infty(\mathbb R)$ provided:
\begin{enumerate}
\item[(a)] case $\mu_2=+\infty$: if $|\omega|<(1-\gamma)min\{1, \frac{|c|}{b}\}$,
\item[(b)] case $\mu_2$ finite: let $\omega$ such that $|\omega|<(1-\gamma)min\{1, \frac{|c|}{b}\}$ and define
\begin{equation}\label{minim}
\alpha_0\equiv \frac1\gamma -|\omega|-\frac{1}{\beta_0}>0
\end{equation}
where $\beta_0\equiv -4\gamma^4(b|\omega|+\frac1\gamma (a-\frac{1}{\gamma^2}))$. Then existence of solitary waves is insured  by choosing $\mu_2$ satisfying $\frac{\sqrt{\mu}}{\sqrt{\mu_2}}<\gamma^2\alpha_0$.
\end{enumerate}
\end{teo}

\begin{obs}\label{rnon} The statement in Theorem \ref{main} deserves to be clarified at  least in some points.
\begin{enumerate}
\item [(1)] The positivity of $\alpha_0$ in (\ref{minim}) is deduced from the condition on the speed velocity $\omega$, $|\omega|<(1-\gamma)min\{1, \frac{|c|}{b}\}$.
\item [(2)] For a fixed $\gamma\in (0,1)$,  in (\ref{BF1}), our approach only gives the existence of solitary waves solutions by choosing the depth parameter $\mu_2$ in an interval of the form $(\mu_0, +\infty)$.
\item[(3)] The regularity of the solitary wave solutions is proved in Theorems \ref{decay} and  \ref{decayfi} below.
\end{enumerate}
\end{obs}

The proof of Theorem \ref{main} is based on the Concentration-Compactness Principle  (see Lions \cite{Lions}). Thus, we consider the family of minimization problems
\begin{equation}\label{I} 
I_\lambda=\inf\{E_{\mu_2}(\xi, \nu): (\xi, \nu)\in H^1(\mathbb R)\times H^1(\mathbb R)\;\;\text{and}\;\; F(\xi, \nu)=\lambda\}
\end{equation}
where $\lambda>0$ and the functionals $E_{\mu_2}$ and $F$ are defined as
\begin{equation}\label{E} 
E_{\mu_2}(\xi, \nu)=\int_{\mathbb R} \frac{(1-\gamma)}{2} \xi J_c \xi+ \frac12 \nu\mathcal L_{\mu_2} v-\omega \xi J_b \nu dx
\end{equation}
and $ F(\xi, \nu)=r \int_{\mathbb R} \xi \nu^2 dx$, with $r=\frac{\epsilon}{2\gamma}$. For the case $\mu_2=+\infty$, $E_{+\infty}$ is defined in the same form as $E_{\mu_2}$ in (\ref{E}) by changing  the operator $\mathcal L_{\mu_2}$ by $\mathcal L_{+\infty}$ defined in (\ref{Linf}).

\vspace{0.3cm}
The paper is organized as follows. In the next section we prove Theorem \eqref{main} in the case $\mu_2=+\infty$ while we treat the case where $\mu_2$ is finite in Section 3. We prove in Section 4 the algebraic decay of the solitary waves.

Finally the last two Sections are devoted to two systems that are the two ways propagation versions of the classical Intermediate Long Wave (ILW) and Benjamin-Ono (BO) equations. Using the parameters introduced above, the ILW regime corresponds to $\mu\sim\epsilon\ll 1$ and $\mu_2\sim 1$ (and thus $\delta^2\sim\mu\sim \epsilon_2).$

If furthermore the lower layer depth is infinite, that is $\delta=0$ (and thus $\mu_2=\infty,\; \epsilon_2=0),$ one obtains the BO regime. Again we refer to \cite{BLS} for details on the justification of those asymptotic regimes.

The Intermediate Long wave (one parameter family of) systems  write in one spatial dimension
\begin{equation}\label{W1}
\left\{\begin{array}{lll}
\mathcal W\partial_t\zeta +\mathcal Z\partial_x v -\frac{\epsilon}{\gamma}\partial_x(\zeta v)=0\\
\partial_t v+ (1-\gamma) \partial_x \zeta -\frac{\epsilon}{2\gamma}\partial_x(v^2)=0,
\end{array}\right.
\end{equation}
where $\mathcal W$ and $\mathcal Z$ are self-adjoint operators defined for $\gamma\in (0,1)$, $\beta> 1$, $\mu, \mu_2,  \varepsilon>0$ by
\begin{equation}\label{W2}
\mathcal W=1+g(D),\qquad \mathcal Z=\frac{1}{\gamma}\Big(1+\frac{\beta-1}{\gamma} \sqrt{\mu} |D| \coth(\sqrt{\mu_2}|D|)\Big ),
\end{equation}
with $|D|=\mathcal H\partial_x$, where $\mathcal H$ represents the Hilbert transform, and 
$$
g(D)=\frac{\beta}{\gamma}\sqrt{\mu} |D| \coth(\sqrt{\mu_2}|D|).
$$

On the other hand the Benjamin-Ono systems write, again in spatial dimension one:

\begin{equation}\label{BO1}
\left\{\begin{array}{lll}
\mathcal D\partial_t\zeta +\mathcal B\partial_x v -\frac{\epsilon}{\gamma}\partial_x(\zeta v)=0\\
\partial_t v+ (1-\gamma) \partial_x \zeta -\frac{\epsilon}{2\gamma}\partial_x(v^2)=0,
\end{array}\right.
\end{equation}
where $\mathcal D$ and $\mathcal B$ are self-adjoint operators defined for $\gamma\in (0,1)$, $\beta> 1$ and $\mu, \varepsilon>0$ by
\begin{equation}\label{BO2}
\mathcal D=1+\frac{\beta}{\gamma} \sqrt{\mu} |D|,\qquad \mathcal B=\frac{1}{\gamma}\Big(1+\frac{\beta-1}{\gamma} \sqrt{\mu} |D|\Big ),
\end{equation}
with $|D|=\mathcal H\partial_x$, where $\mathcal H$ represents the Hilbert transform.

For both systems $\beta>1$ is  a free parameter stemming from the use of the BBM trick.

Long time existence (that is on time scales of order $1/\epsilon$) of the Cauchy problem, in spatial dimension one and two has been obtained for both families of systems by Li Xu (\cite{Xu}).

When one restricts to wave propagating in one direction, all systems \eqref{W1}, resp. \eqref{BO1} yield the Intermediate Long Wave Equation (ILW) (see \cite{KKD})  resp. the Benjamin-Ono equation (BO) (see \cite{B}). Both of those equations are completely integrable and have explicit solitary waves, algebraically decaying at infinity for BO and exponentially for ILW.  We refer to \cite{KS} and the references therein for details. The BO solitary wave is shown to be unique (modulo obvious symmetries) in \cite{AT2}. A similar result for the ILW solitary wave is proven in \cite{AT}.

The asymptotic stability of the BO solitary waves (and of explicit multi-solitons) is established in \cite{KM}. The orbital stability of the ILW solitary wave is proven in \cite{AB, ABH}, see also \cite{An, BS, BSS}. We do not know of asymptotic stability
results for the 1 or N-soliton similar to the corresponding ones for the BO equation.
Those results should be in some sense easier than the corresponding ones for BO
since the exponential decay of the solitons should make the spectral analysis of the
linearized operators easier.

On the other hand we do not know of previous existence results for ILW or BO {\it systems} in (\ref{W1}) and (\ref{BO1}), respectively. In Sections 5 and 6 we prove the existence of even solitary waves for both the systems by implicit functions theorem arguments and by perturbation theory of closed linear operators.

In section 8 we establish some open problems associated to the systems B-FD (\ref{BF1}), (\ref{W1}) and (\ref{BO1}). In Appendix A we prove the global existence of small solutions for (\ref{BF1}) for the cases  $b=d>0$, $a\leqq 0$ and $c<0$.

\vspace{0.3cm}
\noindent{\bf Notations.} 
The norm of  standard Sobolev spaces $H^s(\mathbb R)$ is denoted $||\cdot ||_s$. while the norm of the Lebesgue space $L^2(\mathbb R)$ is denoted $||\cdot||$, $s\geqq 0$. We will denote $|\cdot|_p$ the norm in the Lebesgue space $L^p(\mathbb R),\; 1\leq p\leq \infty$ and  $\langle\cdot , \cdot\rangle$ denotes the scalar product in $L^2(\mathbb R)$. We will denote $\hat {f}$  the Fourier transform of a tempered distribution $f.$ For any $s\in \R,$ we define $|D|^s f$ by its Fourier transform $\widehat{|D|^s f}(y)=|y|^s \hat{f}(y).$ Similarly, we define $L_1=|D| \coth(\sqrt{\mu_2}|D|)$ and $L_2= \coth(\sqrt{\mu_2}|D|)$ by its Fourier transform $\widehat{L_1 f}(y)=|y| \coth(\sqrt{\mu_2}|y|) \hat{f}(y)$ and $\widehat{L_2 f}(y)=\coth(\sqrt{\mu_2}|y|) \hat{f}(y)$, respectively.

\section{Existence of solitary wave solutions for the Full Dispersion-Boussinesq system when $\mu_2=+\infty$}

In  this section we give the proof of the existence of solitary waves solutions of (\ref{soli}) in the case $\mu_2=+\infty$. Thus, we will find a non-trivial solution, $(\xi, \nu)$, for the system
\begin{equation}\label{soli}
 \left\{\begin{array}{lll}
-\omega J_b \xi +\mathcal L_{+\infty}\nu =2r\xi \nu\\
-\omega J_b\nu+ (1-\gamma) J_c\xi =r\nu^2,
\end{array}\right.
\end{equation}
with $\gamma \in (0,1)$ and $r=\frac{\epsilon}{2\gamma}$. 

Our approach for the existence will be via the Concentration-Compactness Principle,  thus we call $\{(\xi_n, \nu_n)\}_{n\geqq 1}$ in $H^1(\mathbb R)\times H^1(\mathbb R)$ a minimizing sequence for $I_\lambda$ if it satisfies $F(\xi_n, \nu_n)=\lambda$ for all $n$ and $\lim_{n\to \infty} E_{+\infty}(\xi_n, \nu_n)= I_\lambda$

For  $ a, b, c$ satisfying the relations in (\ref{abc}), the following Lemma establishes the coercivity property of the functional $E_{+\infty}$.

\begin{prop}\label{propert}
Let $\gamma \in (0,1)$ fixed. Then,  

 \begin{enumerate}
\item[(a)] $E_{+\infty}: H^1(\mathbb R)\times H^1(\mathbb R)\to \mathbb R$ is well defined and satisfies
$$
E_{+\infty}(\xi, \nu)\leqq c_0(\|J_c^{1/2} \xi\|^2 +\|\nu\|_1^2)+c_1(\|J_b^{1/2} \xi\|^2 +\|J_b^{1/2}\nu\|)^2). 
$$
\item[(b)]  For $|\omega|<(1-\gamma)min\{1, \frac{|c|}{b}\}$, we have
\begin{equation}\label{coer}
E_{+\infty}(\xi, \nu)\geqq C \|(\xi, \nu)\|_{1\times 1}^2
\end{equation}
where $C=C(b, c, \mu, |\omega|, \gamma)>0$.
\item[(c)]  $I_\lambda\in (0, +\infty)$.
\item[(d)] All minimizing sequences for $I_\lambda$ are bounded in $H^1(\mathbb R)\times H^1(\mathbb R)$.
\item[(e)]  For all $\theta\in (0,\lambda)$, we have the sub-additivity property of $I_\lambda$,
\begin{equation}\label{sub}
I_\lambda<I_\theta+I_{\lambda-\theta}.
\end{equation}

\end{enumerate}
\end{prop}

The most delicate statement in the Proposition \ref{propert} is item (b). The next Lemma shows firstly that the functional $E_{+\infty}$ is non-negative.

\begin{lema}\label{positi}
Let $\gamma \in (0,1)$ fixed and $a,b,c$ satisfying the relations in (\ref{abc}). Then, for $|\omega|<(1-\gamma)min\{1, \frac{|c|}{b}\}$, we have $E_{+\infty}(\xi, \nu)\geqq 0$ for 
all $(\xi, \nu) \in H^1(\mathbb R)\times H^1(\mathbb R)$.
\end{lema}

\begin{proof} To start with, since $J_b$ is a positive operator we obtain from the Cauchy-Schwarz inequality for positive quadratic forms that
\begin{equation}\label{CS}
\Big |\omega \int_{\mathbb R} \xi J_b \nu dx \Big |\leqq |\omega|\langle J_b \xi, \xi\rangle^{1/2}\langle J_b \nu, \nu\rangle^{1/2}\leqq \frac{|\omega|}{2}\langle J_b \xi, \xi \rangle +\frac{|\omega|}{2}\langle J_b \nu, \nu\rangle.
\end{equation}
Next, we show the following two inequalities under  the condition $|\omega|\leqq (1-\gamma)\min\{1, \frac{|c|}{b}\}$:
\begin{equation}\label{first}
\frac{|\omega|}{2}\langle J_b \xi, \xi \rangle \leqq \frac{(1-\gamma)}{2} \langle J_c \xi, \xi \rangle,\;\;\text{and},\;\;
\frac{|\omega|}{2}\langle J_b \nu, \nu \rangle \leqq \frac{1}{2} \langle \mathcal L_{\mu_2} \nu, \nu \rangle
\end{equation}
We divide the proof in various steps:
\begin{enumerate}
\item[(1)] since the second order polynomial $p(x)= [-(1-\gamma)\mu c-|\omega| \mu b] x^2 + (1-\gamma)-|\omega|$ achieves positive values for all $x\in \mathbb R$, under the conditions $|\omega|\leqq (1-\gamma)$ and $b |\omega|\leqq (1-\gamma)|c|$, we obtain from Plancherel's theorem
$$
\langle [(1-\gamma) J_c -|\omega| J_b] \xi, \xi \rangle=\int_{\mathbb R} |\widehat{\xi}(x)|^2 p(x) dx\geqq 0.
$$

\item[(2)] Since the symbol of $\mathcal L_{\mu_2}-|\omega| J_b$ is given by the even function,
\begin{equation}\label{f}
f(x)= \frac{1}{\gamma}-\frac{\sqrt{\mu}}{\gamma^2}|x|-\mu \Big[ b|\omega|+\frac{1}{\gamma}\Big(a-\frac{1}{\gamma^2}\Big)\Big] x^2,
\end{equation}
we will show that under the condition $|\omega|\leqq (1-\gamma)\min\{1, \frac{|c|}{b}\}$, $\min_{x\in \mathbb R} f(x)=f(x_0)>0$, with $x_0$ being the only positive critical point for $f$ and 
\begin{equation}\label{fmi}
f(x_0)= \frac1\gamma -|\omega|-\frac{1}{4\gamma ^4\beta_0}
\end{equation}
where $\beta_0\equiv -[b|\omega|+\frac1\gamma (a-\frac{1}{\gamma^2})]$ (see (\ref{minim})). In fact, for $x>0$ we have that the only critical point of $f(x)$, $x_0$,  is given by the relation
$$
2\sqrt{\mu}\beta_0x_0=\frac{1}{\gamma^2}.
$$
We note that $\beta_0$ needs to be strictly positive. Moreover, the profile of $f$ is a strictly convex function from the relation $f''(x)=2\sqrt{\mu}\beta_0$. The idea now is to show that $f(x_0)>0$,  which is equivalent to check that the mapping $M: I(0)\to \mathbb R$ defined by (with $\gamma$ fixed)
\begin{equation}\label{M}
M(\omega)= 4(1-\gamma |\omega|)[1-b\gamma ^3 |\omega|-a\gamma^2]-1
\end{equation}
 is strictly positive for $\omega\in I(0)$ satisfying $|\omega|\leqq (1-\gamma)\min\{1, \frac{|c|}{b}\}$. Before, we show that the unique positive critical point $\omega_0$ of $M$ given by
$$
\omega_0=\frac{1+(b-a)\gamma^2}{2b \gamma^3}\equiv q(\gamma)
$$
does not belong to the admissible interval for the speed velocity $\omega$. In fact, note that  $q$ is a strictly decreasing positive function with $\inf_{\gamma \in (0,1)}q(\gamma)=\frac{1-a+b}{2b}$. Now, for $|c|\leqq b$ we have that $\omega_0>\frac{|c|}{b}$, indeed, since
$$
\frac{1-a+b}{2b}>\frac{|c|}{b} \Leftrightarrow \frac23 +3b>-3c
$$
our claim follows immediately. Next, for $|c|\geqq b$ we obtain $\omega_0>1$. Thus, since $M(0)>0$ and $M$ is a  convex function  we observe that in order to show that $M>0$ on $I(0)$ is sufficient to check that $M(1-\gamma)>0$ for $|c|\geqq b$ and $M((1-\gamma)\frac{|c|}{b})>0$ for $|c|\leqq b$. Thus, we consider the case $|c|\leqq b$, then
since $1-\gamma(1-\gamma)\frac{|c|}{b}\geqq 1-\gamma(1-\gamma)$, $1-\frac13(1-\gamma)\gamma^3>0$ and the polynomial $\tau(\gamma)=4(1-\gamma+\gamma^2)\Big [1-\frac13(1-\gamma)\gamma^3\Big ]-1>0$ for all $\gamma\in (0,1)$, we obtain
$$
M\Big ((1-\gamma)\frac{|c|}{b}\Big)\geqq -4a\gamma^2(1-\gamma+\gamma^2)^2+\tau(\gamma)>0.
$$
\end{enumerate}
This concludes the proof of   the Lemma.
\end{proof}

\begin{proof} $[$Proposition \ref{propert}$]$
Item (a) is an immediate consequence of  the Cauchy-Schwarz inequality. For the coercivity property (\ref{coer}), we deduce from (\ref{CS}) 
\begin{equation}\label{coer2}
E_{+\infty}((\xi, \nu))\geqq \frac12 \langle ((1-\gamma)J_c-|\omega| J_b)\xi, \xi\rangle +\frac12 \langle (\mathcal L_{+\infty}-|\omega| J_b)\nu, \nu\rangle \equiv I+II.
\end{equation}
Next, we estimate the term $I$. Suppose $|c|\geqq b$, then $|\omega|<1-\gamma $. Thus
\begin{equation}
\begin{array}{ll}
2I=\int_{\mathbb R} [(1-\gamma)-|\omega|]\xi^2(x)-c\mu (\xi'(x)) ^2((1-\gamma)+\frac{b}{c}|\omega|)dx\\
\geqq [(1-\gamma)-|\omega|]\int_{\mathbb R} \xi^2-c\mu (\xi') ^2dx=[(1-\gamma)-|\omega|]\|J_c^{1/2}\xi\|^2.
\end{array}
\end{equation}
For the term $II$, we have first from  Cauchy-Schwarz inequality and from Young's inequality that for all $\eta>0$,
$$
\|D|^{1/2}\nu\|^2\leqq \|\nu\|\|\nu'\|\leqq \eta\|\nu\|^2+\frac{1}{4\eta}\|\nu'\|^2.
$$
Then, for $-\alpha\equiv \frac{\mu}{\gamma}(a-\frac{1}{\gamma^2})$ we have
\begin{equation}\label{II}
2II\geqq \Big(\frac{1}{\gamma}-|\omega|-\eta \frac{\sqrt{\mu}}{\gamma^2}\Big)\|\nu\|^2+ (\alpha-\mu b|\omega|- \frac{1}{4\eta}\frac{\sqrt{\mu}}{\gamma^2})\|\nu'\|^2.
\end{equation}
Next, for $\beta_0\equiv -[b|\omega|+ \frac{1}{\gamma}(a-\frac{1}{\gamma^2}]>0$ we choose 
\begin{equation}\label{eta}
\eta=  \frac{\gamma^2}{\sqrt{\mu}}\frac{1}{4\gamma^4\beta_0} + \epsilon_0  \frac{\gamma^2}{\sqrt{\mu}}
\end{equation}
for $\epsilon_0$ chosen such that $f(x_0)>\epsilon_0>0$ (see (\ref{fmi})). Thus
\begin{equation}\label{IIa}
\Big(\frac{1}{\gamma}-|\omega|-\eta \frac{\sqrt{\mu}}{\gamma^2}\Big)\|\nu\|^2= (f(x_0)-\epsilon_0)\|\nu\|^2,
\end{equation}
and 
\begin{equation}\label{IIb}
(\alpha-\mu b|\omega|- \frac{1}{4\eta}\frac{\sqrt{\mu}}{\gamma^2})\|\nu'\|^2= \frac{4\mu \epsilon \gamma^4 \beta_0^2}{1+4\epsilon \gamma^4 \beta_0}\|\nu'\|^2\equiv \theta \|\nu'\|^2 .
\end{equation}
Therefore, for $|c|\geqq b$ we get that
\begin{equation}
E_{+\infty}((\xi, \nu))\geqq \frac12 [(1-\gamma)-|\omega|]\|J_c^{1/2}\xi\|^2+\frac12 min\{f(x_0)-\epsilon_0, \theta\} \|\nu\|_1^2.
\end{equation}
We note that the $\|J_c^{1/2}\xi\|$ is equivalent to $\|\xi\|_1$.

Next, we consider the case $|c|\leqq b$. Thus $|\omega|< (1-\gamma)\frac{|c|}{b}$ and the expression $I$ in (\ref{coer2}) is estimated  as
\begin{equation}\label{II3}
\begin{array}{ll}
2I=\int_{\mathbb R} [(1-\gamma)-|\omega|]\xi^2(x)+\mu b(\xi'(x)) ^2 (\frac{|c|}{b}(1-\gamma)-|\omega|)dx\\
\geqq \Big[\frac{|c|}{b}(1-\gamma)-|\omega| \Big]\int_{\mathbb R} \xi^2+\mu b(\xi') ^2dx=\Big[\frac{|c|}{b}(1-\gamma)-|\omega|\Big]\|J_b^{1/2}\xi\|^2.
\end{array}
\end{equation}
The estimative of the term $II$ in (\ref{coer2}) is the same as in the case $|c|\geqq b$. Therefore, for $|c|\leqq b$ we get that
\begin{equation}
E_{+\infty}((\xi, \nu))\geqq \frac12 \Big[\frac{|c|}{b}(1-\gamma)-|\omega|\Big]\|J_b^{1/2}\xi\|^2+\frac12 min\{f(x_0)-\epsilon_0, \theta\} \|\nu\|_1^2.
\end{equation}
We note that  $\|J_b^{1/2}\xi\|$ is equivalent to $\|\xi\|_1$. This concludes item (b).  Now, item (c) follows immediately from the coercivity property (\ref{coer}) and the inequality
$$
0<\lambda=\int_{\mathbb R} \xi \nu^2dx \leqq |\nu|_\infty\|\nu\|\|\xi\|\leqq \|(\xi, \nu)\|_{1\times 1}^3.
$$
Indeed $I_\lambda\geqq C\lambda^{2/3}>0$. Next, item (d) follows immediately from (\ref{coer}). Finally, from the property $I_{\tau \lambda}=\tau^{2/3} I_\lambda$ for all $\tau>0$ and from $I_\lambda>0$ we obtain the sub-additivity property (\ref{sub}). This concludes the proof of the Proposition.

\end{proof}

In order to show the existence of solutions of equation (\ref{soli}), we will prove that a minimizing sequence for the problem (\ref{I}) converges (modulo translations) to a function in $H^1(\mathbb R)\times H^1(\mathbb R)$ satisfying the constraint $F(\xi, \nu)=\lambda$. To do so we will use the following result, it which is the key tool in our analysis  (see Lemma 1.1 in Lions \cite{Lions}).

\begin{lema}\label{lions} $($The Concentration-Compactness Principle $)$\newline
 Let $\{\rho_n\}_{n\geq 1}$ be a sequence of
non-negative functions in $L^1({\Bbb R})$ satisfying
$\int_{-\infty}^{\infty} \rho_n (x)dx=\sigma$ for all $n$ and
some $\sigma >0$. Then there exists a subsequence
$\{\rho_{n_k}\}_{k\geq 1}$ satisfying one of the following three
conditions:

$(1)$\; {\bf{\text{(Compactness)}}} there are $y_k\in {\Bbb R}$
for $k=1,2,...$, such that $\rho_{n_k}(\cdot +y_k)$ is tight, {\it
{i.e.}} for any $\epsilon >0$, there is $R>0$ large enough such
that
$$
\int_{|x-y_k|\leq R}\rho_{n_k}(x)\,dx \geq \sigma - \epsilon ;
$$

$(2)$\; {\bf{\text{(Vanishing)}}} for any $R>0$,
$$
\lim_{k\to \infty}\; \sup_{y\in \Bbb R}\;\int_{|x-y|\leq
R}\rho_{n_k}(x)dx=0;
$$

$(3)$\; {\bf{\text{(Dichotomy)}}} there exists $\theta_0\in
(0,\sigma )$ such that for any $\epsilon >0$, there exists
$k_0\geq 1$ and $\rho_{k,1},\rho_{k,2}\in L^1({\Bbb R})$, with
$\rho_{k,1},\rho_{k,2}\geq 0$, such that for $k\geq k_0$,
$$
\begin{array}{ll}
&|\rho_{n_k}-(\rho_{k,1}+\rho_{k,2})|_{L^1}\leq \epsilon ,\\
\\
&|\int_{\mathbb R}\rho_{k,1}\,dx-\theta_0 |\leq
\epsilon ,\;\;\;\;
|\int_{\mathbb R}\rho_{k, 2}\,dx-(\sigma - \theta_0)
|\leq \epsilon
,\\
\\
&supp\;\rho_{k,1}\cap supp\;\rho_{k,2}=0\!\!\!/ ,\;\;
dist(supp\;\rho_{k,1},\; supp\;\rho_{k,2})\to \infty \; as\;
k\to \infty .
\end{array}
$$
\end{lema}

\begin{obs}\label{rnon} In Lemma \ref{lions} above, the condition
$\int_{\mathbb R} \rho_n(x)\,dx=\sigma$ can be replaced by
$\int_{\mathbb R}\rho_n(x)\,dx$ $=\lambda_n$ where
$\lambda_n\to \sigma
>0$ as $n\to\infty$. It is enough to replace $\rho_n$ by $\sigma\rho_n/\lambda_n$ and
apply the lemma.
\end{obs}

The next step in our analysis is to rule out the possibilities of vanishing and dichotomy in Lemma \ref{lions}. The following classical Lemma (see Lemma I.1 in \cite{Lions}) is the key tool to rule  out vanishing.

\begin{lema}\label{kp}
Let $\{(\xi_n, \nu_n)\}_{n\geqq 1}$, be a bounded sequence in $H^1(\mathbb R)\times H^1(\mathbb R)$. Assume that for some $R>0$,
$$
Q_n(R)=sup_{y\in \mathbb R} \int_{y-R}^{y+R} |\nu_n(x)|^2dx \to 0
$$
as $n\to \infty$. Then, $\int_{\mathbb R} \xi_n(x) \nu_n^2(x)dx\to 0$ as $n\to \infty$.
\end{lema}

\begin{proof}  From H\"older inequality and from the embedding $H^1(\mathbb R)\hookrightarrow L^p(\mathbb R)$, $p\geqq 2$, we have for $C>0$ that
$$
\Big |\int_{\mathbb R} \xi_n(x) \nu_n^2(x)dx\Big |\leqq C\|\xi_n\|_1|\nu_n|^2_{L^3}\leqq C_1\delta_n^{1/3} \|\xi_n\|_1\|\nu_n\|_1^{4/3}
$$
with $\delta_n\to 0$ as $n\to \infty$. This achieves the proof.
\end{proof}

Let  $\{(\xi_n, \nu_n)\}_{n\geqq 1}$ be a minimizing sequence for problem (\ref{I}) and consider the sequence of non-negative functions in $L^1(\mathbb R)$
$$
\rho_n(x)=|\xi_n(x)|^2+ |\xi'_n(x)|^2 +|\nu_n(x)|^2+ |\nu'_n(x)|^2.
$$
Let $\lambda_n=\int_{\mathbb R} \rho_n(x)dx$. Since $\lambda_n=\|(\xi_n, \nu_n)\|^2_{1\times 1}$,   $\lambda_n$ is bounded and $\lambda_n\geqq \lambda^{2/3}$. Assume that $\lambda_n\to \sigma $ as $n\to \infty$. Then, from Lemma \ref{lions}  there exists a subsequence $\{\rho_{n_k}\}_{k\geqq 1}$ of $\{\rho_{n}\}_{n\geqq 1}$ which satisfies either vanishing or dichotomy. If vanishing occurs, then 
$$
\lim_{k\to \infty} sup_{y\in \mathbb R} \int_{y-R}^{y+R} |\nu_{n_k}(x)|^2dx =0, \qquad \text{for any}\;\; R>0,
$$
hence, by Lemma \ref{kp} we obtain $\lim_{k\to \infty} F(\xi_{n_k}, \nu_{n_k})=0$,  which is a contradiction.

If dichotomy occurs, there exist $ \theta\in (0, \sigma)$ and $L^1(\mathbb R)$-functions $\rho_{k,1}, \rho_{k,2}$ satisfying item (3) of Lemma \ref{lions}. Moreover, we may assume that the supports of  $\rho_{k,1}$ and $\rho_{k,2}$ are disjoint as follows:
\begin{equation}\label{suppor}
\text{supp}\;\rho_{k,1}\subset (y_k-R_0, y_k+R_0),\; \text{supp}\;\rho_{k,2}\subset (-\infty, y_k-2R_k)\cup (y_k+2R_k, +\infty)
\end{equation}
for some fixed $R_0$, a sequence $\{y_k\}_{k\geqq 1}$ and $R_k\to +\infty$, as $k\to \infty$.

Now, denoting ${\bf{h}}_{n_k}=(\xi_{n_k}, \nu_{n_k})$ we obtain the following splitting functions ${\bf{h}}_{k,1}$ and ${\bf{h}}_{k, 2}$ of ${\bf{h}}_{n_k}$. Consider $\varphi, \zeta \in C^\infty(\mathbb R)$ with $0\leqq \varphi, \zeta\leqq 1$ such that $\zeta(x)=1$, for $|x|\leqq 1$ and $\zeta(x)=0$ for $|x|\geqq 2$, $\varphi(x)=1$, for $|x|\geqq 2$ and $\varphi(x)=0$ for $|x|\geqq 1$. Denote,
$$
\zeta_k(x)=\zeta\Big(\frac{x-y_k}{R_1}\Big),\;\;\text{and}\;\; \varphi_k(x)=\varphi \Big(\frac{x-y_k}{R_k}\Big),
$$
for $x\in \mathbb R$ and $R_1>R_0$, so that $\text{supp}\; \zeta_k\subset \{x: |x-y_k|\}\leqq 2R_1$ and $\text{supp}\; \varphi_k\subset \{x: |x-y_k|\}\geqq 2R_k$. Now, we define ${\bf{h}}_{k,1}=\zeta_k {\bf{h}}_{n_k}$ and ${\bf{h}}_{k,2}=\varphi_k {\bf{h}}_{n_k}$. Therefore, from the relations
\begin{equation}\label{J}
\begin{array}{ll}
&\int_{y_k-R_0}^{y_k+R_0} |\rho_{n_k}(x)-\rho_{k,1}(x)|dx\leqq\epsilon,\\
\\
&\int_{|x-y_k|\geqq 2R_k} |\rho_{n_k}(x)-\rho_{k,2}(x)|dx\leqq \epsilon,\\
\\
&\int_{R_0\leqq |x-y_k|\leqq 2 R_k} \rho_{n_k}(x)dx\leqq \epsilon,
\end{array}
\end{equation}
$|\zeta_k'|_{\infty}\leqq \frac{1}{R_1}|\zeta|_{\infty}, |\varphi_k'|_{\infty}\leqq \frac{1}{R_1}|\varphi|_{\infty}$, and by using a classical argument we can see that for $R_1$ large enough
\begin{equation}\label{J0}
N_0=\Big |\|\zeta_k \xi_{n_k}\|^2_1+ |\zeta_k \nu_{n_k}\|^2_1-\int_{\mathbb R} \rho_{k,1}(x)dx\Big|\leqq \epsilon,
\end{equation}
and
\begin{equation}\label{J1}
N_1=\Big |\|\varphi_k \xi_{n_k}\|^2_1+ \| \varphi_k\nu_{n_k}\|^2_1-\int_{\mathbb R} \rho_{k,2}(x)dx\Big|\leqq \epsilon.
\end{equation}

Next, we define ${\bf{w}}_{k}= {\bf{h}}_{n_k}-({\bf{h}}_{k,1}+{\bf{h}}_{k, 2})$. Then, for $\chi_k=1-\zeta_k-\varphi_k$ we obtain
\begin{equation}\label{w}
\|{\bf{w}}_{k}\|_{1\times 1}=O(\epsilon).
\end{equation}
Indeed, since $\text{supp}\; \chi_k\subset \{x: R_1\leqq |x-y_k|\}\leqq 2R_k$ we obtain from (\ref{J}) that
$$
\|\chi_k \xi_k\|^2_1\leqq 2 (|\chi_k|^2_\infty+ |\chi'_k|^2_\infty)\int_{R_1\leqq |x-y_k|\leqq 2 R_k} \rho_{n_k}(x)dx= O(\epsilon).
$$
Similarly, we obtain $\|\chi_k \nu_k\|^2_1=O(\epsilon)$.  Next, since $\int_{\mathbb R} \zeta_k^3 \xi_{n_k}\nu_{n_k}^2dx$ is bounded, there is a subsequence of $\{{\bf{h}}_{k,1}\}$, still denoted by $\{{\bf{h}}_{k,1}\}$, such that for $k\geqq k_0>0$ and $\theta\in \mathbb R$
\begin{equation}\label{thet}
\Big |\int_{\mathbb R} \zeta_k^3 \xi_{n_k}\nu_{n_k}^2dx-\theta \Big | \leqq \epsilon.
\end{equation}
Moreover, since $F(\xi_{n_k}, \nu_{n_k})=\lambda$, $\zeta_k\varphi_k=0$, $\chi_k^3=1-\zeta_k^3-\varphi_k^3$ we obtain for $k\geqq k_0$,
\begin{equation}\label{thet2}
\Big |\int_{\mathbb R} \varphi_k^3 \xi_{n_k}\nu_{n_k}^2dx-(\lambda-\theta) \Big | \leqq \Big |\int_{\mathbb R} \chi_k^3 \xi_{n_k}\nu_{n_k}^2dx\Big |+\epsilon\leqq \|{\bf{w}}_{k}\|^3_{1\times 1} + \epsilon=O(\epsilon).
\end{equation}

Next, we will show that for $R_1, R_k$, large enough,  we have 
\begin{equation}\label{key1}
E_{+\infty}({\bf{h}}_{n_k})=E_{+\infty}({\bf{h}}_{k,1}) + E_{+\infty}({\bf{h}}_{k,2}) + O(\epsilon).
\end{equation}
Indeed, denoting  $(a_k, b_k)\equiv {\bf{w}}_{k}=(\chi_k \xi_{n_k}, \chi_k\nu_k)$  one easily checks  that 
\begin{equation}\label{key2}
E_{+\infty}({\bf{h}}_{n_k})=E_{+\infty}({\bf{w}}_{k})+E_{+\infty}({\bf{h}}_{k,1}) + E_{+\infty}({\bf{h}}_{k,2}) + N_3+N_4,
\end{equation}
where for positive constants $c_1, c_2, c_3, c_5, c_6$ and $c_4<0$ 
\begin{equation}\label{N_3}
\begin{array}{ll}
N_3=&\int_{\mathbb R} (\zeta_k+\varphi_k)[(2 c_1a_k-\omega b_k) \xi_{n_k} + (c_2b_k-\omega a_k) \nu_{n_k}]\\
\\
&+ ((\zeta_k \xi_{n_k})' + (\varphi_k \xi_{n_k})' )
(2 c_3a'_k-\omega c_4 b'_k)+((\zeta_k \nu_{n_k})' +(\varphi_k \nu_{n_k})') (2 c_5b'_k-\omega c_4 a'_k)
\end{array}
\end{equation}
and
\begin{equation}\label{N_4}
N_4=-2c_6\int_{\mathbb R}  [\zeta_k \nu_{n_k} +\varphi_k \nu_{n_k}]D(b_k) + \varphi_k \nu_{n_k}D(\zeta_k \nu_{n_k})dx.
\end{equation}
Moreover  $N_3=O(\epsilon)$ and  $N_4=O(\epsilon)$.  Indeed, from Cauchy-Schwarz inequality, Proposition \ref{propert} (item (d)) and (\ref{w}) we obtain for $R_k$ large enough  that
$$
|N_3|\leqq C\| {\bf{w}}_{k}\|_{1\times 1}\|{\bf{h}}_{n_k}\|_{1\times 1}=O(\epsilon).
$$
Now, concerning $N_4$ we observe first from Cauchy-Schwarz inequality, Proposition \ref{propert} and (\ref{w}),   that
\begin{equation}\label{1N_4}
\int_{\mathbb R}  [\zeta_k \nu_{n_k} +\varphi_k \nu_{n_k}]D(b_k) dx\leqq 2\|D(b_k)\|\|\nu_{n_k} \|\leqq C \| {\bf{w}}_{k}\|_{1\times 1}\|{\bf{h}}_{n_k}\|_{1\times 1}=O(\epsilon).
\end{equation}
Next, since $D=\mathcal H \partial_x$, where $\mathcal H$ is the Hilbert-transform, we have by using commutators, $\zeta_k \varphi_k=0$ and Calderon commutator Theorem (see \cite{C}-\cite{CM}) that for $R_1$ large
\begin{equation}\label{2N_4}
\begin{array}{ll}
\Big |\int_{\mathbb R}  \varphi_k \nu_{n_k}D(\zeta_k \nu_{n_k})dx\Big |&=\Big |\int_{\mathbb R} \varphi_k \nu_{n_k}(\mathcal H(\zeta'_k \nu_{n_k}) + [\mathcal H,\zeta_k]\nu'_{n_k} + \zeta_k \mathcal H \nu'_{n_k})dx \Big|\\
&\leqq C\|\nu_{n_k}\|( [\mathcal H,\zeta_k]\nu'_{n_k} \| +\|\zeta'_k \nu_{n_k}\|)\\
&\leqq C\|\nu_{n_k}\|^2|\zeta'_k|_{\infty}=C\frac{|\zeta'|_{\infty}}{R_1}\|\nu_{n_k}\|^2=O(\epsilon).
\end{array}
\end{equation}
Thus, from (\ref{1N_4})-(\ref{2N_4}) we obtain $N_4=O(\epsilon)$.

Now, from  Proposition \ref{propert}, item (a), and (\ref{w}) it follows that for $R_1, R_k$ large enough, $
E_{+\infty}({\bf{w}}_{k})\leqq C\|{\bf{w}}_{k}\|_{1\times 1}=O(\epsilon)$, and so we obtain the key relation (\ref{key1}). Moreover,
\begin{equation}\label{3key}
I_\lambda\geqq \liminf_{k\to \infty} E_{+\infty}({\bf{h}}_{n_k})\geqq \liminf_{k\to \infty} E_{+\infty}({\bf{h}}_{k,1}) + \liminf_{k\to \infty} E_{+\infty}({\bf{h}}_{k,2}) + O(\epsilon).
\end{equation}

The next step in our analysis is to rule out dichotomy. Indeed,   recalling that
$\theta=\lim_{k\to\infty} \int_{\mathbb R} \zeta_k^3 \xi_{n_k}\nu_{n_k}^2dx$, we consider the following cases:

\vspace{0.3cm}
\begin{enumerate}
\item[1)] Suppose $\theta=0$. Then, from (\ref{thet2}) we have for $k$ large that $F({\bf{h}}_{k,2})>\lambda/2>0$. Therefore, for $k$ fixed we consider $d_k>0$ such that $F(d_k{\bf{h}}_{k,2})=\lambda$. Moreover,
$$
|d_k-1|\leqq \Big(\frac{2}{\lambda}\Big)^{1/3}|\lambda^{1/3}-F({\bf{h}}_{k,2})^{1/3}|\leqq C_1 \epsilon
$$
with $C_1$ independent of ${\bf{h}}_{k,2}$ and  $\epsilon$. Hence, $d_k\to 1$, as $k\to \infty$, and 
\begin{equation}\label{5key}
I_\lambda\leqq E_{+\infty}(d_k {\bf{h}}_{k,2})=d^2_k E_{+\infty}({\bf{h}}_{k,2})= E_{+\infty}({\bf{h}}_{k,2})+ O(\epsilon).
\end{equation}
Thus, from the coercivity property of $E_{+\infty}$ (see item (b)-Proposition \ref{propert}) and from (\ref{J0}) and item (3) in Lemma \ref{lions} we obtain for $\theta_0>0$,
$$
\liminf_{k\to \infty} E_{+\infty}({\bf{h}}_{k,1})\geqq C \liminf_{k\to \infty} \|{\bf{h}}_{k,1}\|^2_{1\times 1}\geqq C\liminf_{k\to \infty} |\rho_{k,1}|_{L^1} + O(\epsilon)\geqq C\theta_0 + O(\epsilon)
$$
and therefore from (\ref{3key})-(\ref{5key})
$$
I_\lambda\geqq C\theta_0 + I_\lambda +O(\epsilon).
$$
Thus, letting $\epsilon\to 0$ in the last relation leads to the contradiction $
I_\lambda\geqq C\theta_0 + I_\lambda$.

\item[2)] Suppose $\lambda>\theta>0$.  Using the previous argument combined with item (3) in Lemma \ref{lions} and (\ref{J1}), we can prove that $I_\lambda\geqq I_\theta + I_{\lambda-\theta} + O(\epsilon)$ and for $\epsilon\to 0$ we obtain 
$$
I_\lambda\geqq I_\theta + I_{\lambda-\theta}.
$$
But, from (\ref{sub}) we obtain a contradiction.
\item[3)] Case $\theta<0$.  From (\ref{thet2}) follows that $\lim_{k\to \infty} F({\bf{h}}_{k,2})=\lambda-\theta>\lambda/2$. Thus, by a  similar analysis as in item 1) we get a contradiction.

\item[4)] Case $\theta=\lambda$. Then $F({\bf{h}}_{k,1})>\lambda/2$ for $k$ large, and so by using (\ref{J1}) and again,  an  analysis as in item 1) leads to a contradiction.

\item[5)] Case  $\theta>\lambda$. Let consider $e_k>0$ such that $F(e_k{\bf{h}}_{k,1})=\lambda$. Hence, $e_k\to (\lambda/\theta)^{1/3}$, as $k\to \infty$. Then, for $k$ large and from the positivity of $E_{+\infty}$
\begin{equation}\label{6key}
I_\lambda\leqq E_{+\infty}(e_k {\bf{h}}_{k,1})=e^2_k E_{+\infty}({\bf{h}}_{k,1}) <E_{+\infty}({\bf{h}}_{k,1}).
\end{equation}
Thus, from the coercivity property of $E_{+\infty}$, from (\ref{J1}) and item (3) in Lemma \ref{lions} we obtain for $\theta_0>0$
$$
\liminf_{k\to \infty} E_{+\infty}({\bf{h}}_{k,2})\geqq C \liminf_{k\to \infty} \|{\bf{h}}_{k,2}\|^2_{1\times 1}\geqq C\liminf_{k\to \infty} |\rho_{k,2}|_{L^1} + O(\epsilon)\geqq C\theta_0 + O(\epsilon)
$$
and therefore from (\ref{3key})-(\ref{6key})
$$
I_\lambda\geqq C\theta_0 + I_\lambda +O(\epsilon).
$$
Thus, letting $\epsilon\to 0$ in the last relation leads to the contradiction $
I_\lambda\geqq C\theta_0 + I_\lambda$.
\end{enumerate}

Since vanishing and dichotomy have been ruled out, if follows from Lemma \ref{lions} that there exists a sequence $\{y_k\}_{k\geqq 1}\subset {\Bbb R}$  such that for any $\epsilon >0$, there are $R>0$, large enough, and $k_0>0$ such that for $k\geqq k_0$
$$
\int_{|x-y_k|\leq R}\rho_{n_k}(x)\,dx \geq \sigma - \epsilon, \quad \int_{|x-y_k|\geqq R}\rho_{n_k}(x)\,dx =O(\epsilon),
$$
and  for $P_k=(-\infty, y_k-R]\cup [y_k+R, \infty)$
$$
\Big| \int_{P_k} \zeta_{n_k} \nu^2_{n_k} dx \Big |\leqq |\zeta_{n_k}|_{L^3} \Big (\int_{P_k} |\nu_{n_k}|^3 dx\Big )^{2/3}\leqq C\|\zeta_{n_k}\|_1|\nu_{n_k}|^{2/3}_{\infty}\Big(\int_{P_k} \rho_{n_k}(x)\;dx\Big )^{2/3}=O(\epsilon),
$$
where we have used that every minimizing sequence is bounded in $H^1(\mathbb R)\times H^1(\mathbb R)$. 

Then, it follows that
\begin{equation}\label{2const}
\Big |\int_{|x-y_k|\leq R} \zeta_{n_k}(x) \nu^2_{n_k} (x)\,dx-\lambda\Big |\leqq   \epsilon
\end{equation}

Let $\tilde{{\bf{h}}}_{n_k}(x)=(\tilde{\zeta}_{n_k}(x),\tilde{\nu}_{n_k}(x))=(\zeta_{n_k}(x-y_k), \nu_{n_k}(x-y_k))$. Then $\{\tilde{{\bf{h}}}_{n_k}\}_{k\geqq 1}$ is bounded in $H^1(\mathbb R)\times H^1(\mathbb R)$ and therefore $\{\tilde{{\bf{h}}}_{n_k}\}_{k\geqq 1}$  (up to  a subsequence) converges weakly in $H^1(\mathbb R)\times H^1(\mathbb R)$ to a vector-valued function $\tilde{{\bf{h}}}=(\zeta_0, \nu_0)$. Then it follows from (\ref{2const}) that for $k\geqq k_0$,
\begin{equation}\label{3const}
\lambda\geqq \int_{-R}^{R} \tilde\zeta_{n_k}(x) \tilde\nu^2_{n_k}(x)\,dx\geqq \lambda-\epsilon,
\end{equation}
and so by the compact embedding of $H^1(-R, R)$ in $L^2(-R, R)$ we  have the relation $\lambda\geqq \int_{-R}^{R} \zeta_0(x) \nu^2_{0}(x)\,dx\geqq \lambda-\epsilon$. Thus for $\epsilon=\frac1j$, $j\in \mathbb N$, there exists $R_j>j$ such that
\begin{equation}\label{4const}
\lambda\geqq \int_{-R_j}^{R_j}  \zeta_0(x) \nu^2_{0}(x)\,dx\geqq \lambda-\frac1j,
\end{equation}
and consequently as $j\to\infty$, we finally have that $F(\zeta_0, \nu_{0})=\lambda$. Furthermore, from the weak lower semicontinuity of $E_{+\infty}$ and the invariance of $E_{+\infty}$ by translations, we have
$$
I_\lambda=\liminf_{k\to \infty} E_{+\infty}(\tilde{\zeta}_{n_k},\tilde{\nu}_{n_k})\geqq E_{+\infty}(\zeta_0, \nu_{0})\geqq I_\lambda.
$$
Thus, the vector-valued function $\tilde{{\bf{h}}}=(\zeta_0, \nu_0)$ solves the variational problem (\ref{I}) and therefore there exists $K>0$ (Lagrange multiplier) such that
\begin{equation}\label{lagran}
 \left\{\begin{array}{lll}
(1-\gamma) J_c\xi_0-\omega J_b\nu_0=K r \nu^2_0,\\
\mathcal L_{+\infty}\nu_0 -\omega J_b \xi_0 =2Kr\xi \nu\\
\end{array}\right.
\end{equation}
with $r=\frac{\varepsilon}{2\gamma}$. We note that $K>0$ since $3K\lambda=2E_{+\infty}(\zeta_0, \nu_{0})=I_\lambda>0$. Thus, if we consider $(\zeta, \nu)=K(\zeta_0, \nu_{0})$ then from (\ref{lagran}) we have that $(\zeta, \nu)$ solves problem (\ref{soli}). This concludes part (a) of Theorem \ref{main}.

\section{Existence of solitary wave solutions when $\mu_2$ is  finite}

In  this section we give the proof of the existence of solitary waves solutions for system (\ref{0soli}) when  $\mu_2$ is finite and satisfies appropriate conditions. We will find a non-trivial solution $(\xi, \nu)$ by solving the minimization problem (\ref{I}).

We start the analysis by establishing the coercivity property  of $E_{\mu_2}$. For  $ a, b, c$ satisfying the relations in (\ref{abc}) with $b=d$ we obtain:

\begin{prop}\label{propert2}
Let $\gamma \in (0,1)$ fixed and $|\omega|<(1-\gamma)min\{1, \frac{|c|}{b}\}$. Then, for $(\xi, \nu)\in H^1(\mathbb R)\times H^1(\mathbb R)$, we have
\begin{enumerate}
\item[(a)] $E_{\mu_2}(\xi, \nu)$ is well defined and satisfies
$$
E_{\mu_2}(\xi, \nu)\leqq c_0(\|J_c^{1/2} \xi\|^2 +\|\nu\|_1^2)+c_1(\|J_b^{1/2} \xi\|^2 +\|J_b^{1/2}\nu\|)^2). 
$$
\item[(b)] $\langle ((1-\gamma) J_c-|\omega| J_b)\xi, \xi\rangle \geqq 0$.
\item[(c)] Let $\alpha_0>0$ defined in (b)-Theorem \ref{main}. Then for every fixed $\gamma$, we can choose $\mu_2$ satisfying $\frac{\sqrt{\mu}}{\sqrt{\mu_2}}<\gamma^2 \alpha_0$ and such that
\begin{equation}\label{coer3}
\langle (\mathcal L_{\mu_2}-|\omega| J_b)\nu, \nu\rangle \geqq 0.
\end{equation}
\item[(d)]  For $C=C(b, c, \mu, \mu_2, |\omega|, \gamma)>0$, we have
\begin{equation}\label{coer4}
E_{\mu_2}(\xi, \nu)\geqq C \|(\xi, \nu)\|_{1\times 1}^2.
\end{equation}
\item[(e)]  $I_\lambda\in (0, +\infty)$.
\item[(f)] All minimizing sequences for $I_\lambda$ are bounded in $H^1(\mathbb R)\times H^1(\mathbb R)$.
\item[(g)]  For all $\theta\in (0,\lambda)$, we have the sub-additivity property of $I_\lambda$,
\begin{equation}\label{sub2}
I_\lambda<I_\theta+I_{\lambda-\theta}.
\end{equation}

\end{enumerate}
\end{prop}

\begin{proof} The proof of items (a)-(b)-(e)-(f)-(g) are similar to those given in Proposition \ref{propert}. For inequality (\ref{coer3}), we consider the symbol associated to the linear operator $\mathcal L_{\mu_2}-|\omega| J_b$ given by the following even function
$$
f_{\mu_2}(x)= \frac{1}{\gamma}-\frac{\sqrt{\mu}}{\gamma^2} |x|\coth(\sqrt{\mu_2} |x|)- \frac{\mu}{\gamma}\Big(a- \frac{1}{\gamma^2}\coth^2(\sqrt{\mu_2} |x|)\Big ) x^2- |\omega|-\mu b|\omega| x^2. 
$$
Next, since for $x>0$ we observe that
\begin{equation}\label{coth}
x\;\text{coth}(\sqrt{\mu_2}x )\leqq \frac{1}{\sqrt{\mu_2}} +x, 
\end{equation}
and that for all $x$, $\text{coth}^2(\sqrt{\mu_2}|x| )-1\geqq 0$.

Then  it follows that
\begin{equation}\label{coth2}
f_{\mu_2}(x)\geqq f_{+\infty}(x) - \frac{\sqrt{\mu}}{\gamma^2}\frac{1}{\sqrt{\mu_2}}\geqq f_{+\infty}(x_0) - \frac{\sqrt{\mu}}{\gamma^2}\frac{1}{\sqrt{\mu_2}}>0,
\end{equation}
where the positivity property is deduced from (\ref{fmi}) and from the condition on $\mu_2$.

For the coercivity property (\ref{coer4}), we only need to show that
\begin{equation}\label{coth3}
\langle (\mathcal L_{\mu_2}-|\omega| J_b)\nu,\nu\rangle\geqq C\|\nu\|^2_1.
\end{equation}
Indeed,
\begin{equation}\label{coth4}
\begin{array}{lll}
\langle (\mathcal L_{\mu_2}-|\omega| J_b)\nu,\nu\rangle&=\int_{\mathbb R}(\Big(\frac1\gamma-|\omega|\Big)\nu^2(x)- \frac{\sqrt{\mu}}{\gamma^2}\nu (x) |D|\coth(\sqrt{\mu_2} |D|)\nu(x) \\
\\
&+ \frac{\mu}{\gamma}\nu(x)\Big(a- \frac{1}{\gamma^2}\coth^2(\sqrt{\mu_2} |D|)\Big ) \partial_x^2 \nu (x)+ \mu b|\omega| \nu(x) \partial_x^2\nu(x) dx. 
\end{array}
\end{equation}
Next, we estimative the integral terms in (\ref{coth4}). From (\ref{coth}) and Plancherel Theorem we obtain
\begin{equation}\label{coth5}
-\int_{\mathbb R} \frac{\sqrt{\mu}}{\gamma^2}\nu (x) |D|\coth(\sqrt{\mu_2} |D|)\nu(x)\geqq - \frac{\sqrt{\mu}}{\gamma^2} \frac{1}{\sqrt{\mu_2}}\|\nu\|^2- \frac{\sqrt{\mu}}{\gamma^2}\||D|^{1/2} \nu\|^2.
\end{equation}
Now,
\begin{equation}\label{coth6}
\begin{array}{lll}
\int_{\mathbb R} \nu(x)\Big(a- &\frac{1}{\gamma^2}\coth^2(\sqrt{\mu_2} |D|)\Big ) \partial_x^2 \nu (x)dx\\
&=-\Big(a- \frac{1}{\gamma^2}\Big )\|\nu'\|^2 +\frac{1}{\gamma^2} \int_{\mathbb R} \nu(x)[1- \coth^2(\sqrt{\mu_2} |D|)]\partial_x^2\nu (x)dx\\
&\geqq -\Big(a- \frac{1}{\gamma^2}\Big )\|\nu'\|^2, 
\end{array}
\end{equation}
where we have used that $\int_{\mathbb R} \nu(x)[1- \coth^2(\sqrt{\mu_2} |D|)]\partial_x^2\nu (x)dx\geqq 0$. 

Thus, by following a similar analysis as in the proof of (\ref{coer}) we obtain for $\eta$ defined in (\ref{eta}) that
\begin{equation}\label{coth7}
\begin{array}{lll}
\langle (\mathcal L_{\mu_2}-|\omega| J_b)\nu,\nu\rangle&= \Big(\frac{1}{\gamma}-|\omega|
-\frac{\sqrt{\mu}}{\gamma^2}\frac{1}{\sqrt{\mu_2}}-\eta \frac{\sqrt{\mu}}{\gamma^2}\Big)\|\nu\|^2
+(\alpha-\mu b|\omega|- \frac{1}{4\eta}\frac{\sqrt{\mu}}{\gamma^2})\|\nu'\|^2\\
& = r\|\nu\|^2 + \theta \|\nu'\|^2,
\end{array}
\end{equation}
where $\alpha\equiv -\frac{\mu}{\gamma}(a-\frac{1}{\gamma^2})$, $\theta>0$ is defined in (\ref{IIb}) and $\epsilon_0$ is chosen such that
$$
r=f(x_0)-\frac{\sqrt{\mu}}{\gamma^2}\frac{1}{\sqrt{\mu_2}}-\epsilon_0>0.
$$

Thus, from the analysis above, we obtain the coercivity property of  $E_{\mu_2}$, 
\begin{equation}\label{coer5}
2E_{\mu_2}(\xi, \nu)\geqq min \Big \{(1-\gamma)-|\omega|, (1-\gamma)\frac{|c|}{b}-|\omega|\Big \}\|J^{1/2}_b\xi\|^2 +min\{r, \theta\} \|\nu\|_1^2.
\end{equation}
This ends   the proof of the Proposition.
\end{proof}

The proof of existence of non-trivial solutions of equation (\ref{0soli}) follows the same strategy as in the case $\mu_2=+\infty$. We use the Concentration-Compactness's Lemma \ref{lions} with the $H^1(\mathbb R)\times H^1(\mathbb R)-\rho_{n}$  type sequence defined in Section 2 to show the existence of a minimum of (\ref{I}). The analysis for  ruling-out vanishing and dichotomy follows {\it mutatis mutandis} the case of an infinite depth $\mu_2$ and  the more delicate part to be   obtained is  the  key relation (\ref{key1}) for the functional $E_{\mu_2}$. 

Thus, by using the same notations as in Section 2, we have that for  $(a_k, b_k)\equiv {\bf{w}}_{k}=(\chi_k \xi_{n_k}, \chi_k\nu_{n_k})$
\begin{equation}\label{key3}
E_{\mu_2}({\bf{h}}_{n_k})=E_{\mu_2}({\bf{w}}_{k})+E_{\mu_2}({\bf{h}}_{k,1}) + E_{\mu_2}({\bf{h}}_{k,2}) + N_5+N_6 +N_7,
\end{equation}
where for positive constants $c_1, c_2, c_3, c_6$, and $c_4<0$,
\begin{equation}\label{N_5}
\begin{array}{ll}
N_5=&\int_{\mathbb R} (\zeta_k+\varphi_k)[(2 c_1a_k-\omega b_k) \xi_{n_k} + (c_2b_k-\omega a_k) \nu_{n_k}]\\
\\
&+ ((\zeta_k \xi_{n_k})' + (\varphi_k \xi_{n_k})' )
(2 c_3a'_k-\omega c_4 b'_k)-\omega c_4 a'_k ((\zeta_k \nu_{n_k})' +(\varphi_k \nu_{n_k})'),
\end{array}
\end{equation}
\begin{equation}\label{N_6}
N_6=-2c_6\int_{\mathbb R}  b_k |D|\coth(\sqrt{\mu_2}|D|)(\zeta_k \nu_{n_k} +\varphi_k \nu_{n_k}) + \zeta_k \nu_{n_k} |D|\coth(\sqrt{\mu_2}|D|)(\varphi_k \nu_{n_k})dx,
\end{equation}
and for $J(|D|)= a-\frac{1}{\gamma^2} \coth^2 (\sqrt{\mu_2} |D|)$ and $c_7>0$ we have
\begin{equation}
N_7=c_7\int_{\mathbb R}  b_k J(|D|)\partial_x^2 (\zeta_k \nu_{n_k} +\varphi_k \nu_{n_k}) + \zeta_k \nu_{n_k}  J(|D|)\partial_x^2(\varphi_k \nu_{n_k})dx
\end{equation}
and such that $N_5=O(\epsilon)$, $N_5=O(\epsilon)$ and $N_7=O(\epsilon)$. 

 Indeed, from Cauchy-Schwarz inequality, Proposition \ref{propert} (item (d)) and (\ref{w}) we obtain for $R_k$ large that
$$
|N_5|\leqq C\| {\bf{w}}_{k}\|_{1\times 1}\|{\bf{h}}_{n_k}\|_{1\times 1}=O(\epsilon).
$$
Now, with regard to $N_6$ we observe first from relation (\ref{coth}), Cauchy-Schwarz inequality, Proposition \ref{propert} and (\ref{w}) that the first  integral term in (\ref{N_6}) can be estimated as
\begin{equation}\label{1N_6}
\begin{array}{ll}
\int_{\mathbb R}  b_k &|D|\coth(\sqrt{\mu_2}|D|)(\zeta_k \nu_{n_k} +\varphi_k \nu_{n_k}) dx=\int_{\mathbb R}  b_k (|D|\coth(\sqrt{\mu_2}|D|)-|D|)[(\zeta_k +\varphi_k) \nu_{n_k}] \\
\\
&+ \int_{\mathbb R} \nu_{n_k}(\zeta_k +\varphi_k) |D|b_k dx\\
\\
&\leqq \frac{2}{\sqrt{\mu_2}}\int_{\mathbb R} |b_k| |\nu_{n_k}| dx+ 2 \int_{\mathbb R} ||D|b_k| |\nu_{n_k}|dx\leqq C\|{\bf{h}}_{n_k}\|_{1\times 1}\|{\bf{w}}_{k}\|_{1\times 1}=O(\epsilon).
\end{array}
\end{equation}
Next, we define the bounded linear operator $\mathcal T: L^2(\mathbb R)\to L^2(\mathbb R)$  by
\begin{equation}
\mathcal T= \coth(\sqrt{\mu_2}|D|)-\frac{1}{\sqrt{\mu_2}|D|},
\end{equation}
and so since $\zeta_k\varphi_k\equiv 0$ we have that the second integral term in (\ref{N_6}) can be writen as
\begin{equation}\label{2N_6}
\begin{array}{ll}
\int_{\mathbb R} \zeta_k \nu_{n_k} |D|coth(\sqrt{\mu_2}|D|)(\varphi_k \nu_{n_k})dx&=\int_{\mathbb R} \zeta_k \nu_{n_k} \Big[|D|\coth(\sqrt{\mu_2}|D|)-\frac{1}{\sqrt{\mu_2}}\Big](\varphi_k\nu_{n_k})dx\\
\\
&= \int_{\mathbb R} \zeta_k \nu_{n_k} \mathcal T|D|(\varphi_k\nu_{n_k})dx\equiv L.
\end{array}
\end{equation}
Now, we  estimate $L$. Since $D=\mathcal H \partial_x$, we have by using commutators,
\begin{equation}\label{3N_6}
L=  \int_{\mathbb R}  \zeta_k \nu_{n_k} \mathcal T(\varphi_k \mathcal H \nu'_{n_k}) +\zeta_k \nu_{n_k} 
\mathcal T[\mathcal H, \varphi_k] \nu'_{n_k} +  \zeta_k \nu_{n_k} 
\mathcal T \mathcal H(\varphi_k' \nu_{n_k}) dx\equiv L_1+L_2+L_3.
\end{equation}
Next, since the symbol, $m$, associated with  the nonlocal operator $\mathcal T$,
$$
m(x)=\coth(\sqrt{\mu_2}|x|)-\frac{1}{\sqrt{\mu_2}|x|},
$$
satisfies that $m\in C^\infty(\mathbb R-\{0\})$ and for all $x\neq  0$
$$
\Big |\frac{d^k}{dx^k}m(x)\Big |\leqq C_k\frac{1}{|x|^k}
$$
with $k\geqq 0$, it follows from Coifman-Meyer Theorem (see \cite{CM}) and from $\zeta_k\varphi_k\equiv 0$ that for $R_k$ large,
\begin{equation}\label{4N_6}
L_1=  \int_{\mathbb R}  \zeta_k \nu_{n_k} [\mathcal T,\varphi_k ]\mathcal H \nu'_{n_k}\leqq \|\nu_{n_k}\|\|[\mathcal T,\varphi_k ]\mathcal H \nu'_{n_k}\|\leqq C|\varphi'_k|_{\infty}\|\nu_{n_k}\|^2\leqq C \frac{|\varphi'|_{\infty}}{R_k} \|{\bf{h}}_{n_k}\|_{1\times 1}^2=O(\epsilon).
\end{equation}
Next, from the bounded property of $\mathcal T$ and Calderon commutator Theorem (\cite{C}) we obtain for $R_k$ large
\begin{equation}\label{5N_6}
L_2\leqq  \|\nu_{n_k}\|\|[\mathcal H,\varphi_k ] \nu'_{n_k}\|\leqq C|\varphi'_k|_{\infty}\|\nu_{n_k}\|^2=O(\epsilon).
\end{equation}
Similarly, we obtain
\begin{equation}\label{6N_6}
L_3\leqq  \|\nu_{n_k}\|\|\varphi_k\nu_{n_k}\|\leqq C|\varphi'_k|_{\infty}\|\nu_{n_k}\|^2=O(\epsilon).
\end{equation}
Thus, from (\ref{4N_6})-(\ref{5N_6})-(\ref{6N_6}) we obtain $L=O(\epsilon)$.

Now, we estimate the  term $N_7$. By denoting $J_0(|D|)=1-\coth^2(\sqrt{\mu_2}|D|)$, we deduce that  the linear operator $J_0(|D|)\partial_x^2$ is $L^2(\mathbb R)$-bounded
since the associated symbol $(\coth^2(\sqrt{\mu_2}|y|)-1)y^2$  is bounded on the real line,  and so from Cauchy-Schwarz inequality we obtain
\begin{equation}\label{1N_7}
\begin{split}
\int_{\mathbb R}  b_k J(|D|)\partial_x^2 (\zeta_k \nu_{n_k} &+\varphi_k \nu_{n_k}) dx = \int_{\mathbb R}  (a-\frac{1}{\gamma^2}) b_k \partial_x^2(\zeta_k \nu_{n_k}+ \varphi_k \nu_{n_k}) dx\\
&+ \frac{1}{\gamma^2}\int_{\mathbb R} b_kJ_0(|D|)\partial_x^2(\zeta_k \nu_{n_k} +\varphi_k \nu_{n_k}) dx\\
&\leqq C ( \|b'_k\| \| \nu'_{n_k}\| + \|b_k\| \| \nu_{n_k}\|)\leqq C\| {\bf{w}}_{k}\|_{1\times 1}\|{\bf{h}}_{n_k}\|_{1\times 1}=O(\epsilon).
\end{split}
\end{equation}

Next, we define the bounded linear operator $\mathcal T_1: L^2(\mathbb R)\to L^2(\mathbb R)$  by
\begin{equation}\label{T_1}
\mathcal T_1= \coth^2(\sqrt{\mu_2}|D|)-\frac{1}{\mu_2 |D|^2},
\end{equation}
with  symbol $m_1$, defined by $
m_1(x)=\coth^2(\sqrt{\mu_2}|x|)-\frac{1}{\mu_2|x|^2}$.

Then, since $m_1\in C^\infty(\mathbb R-\{0\})$ and for all $x\neq  0$
$$
\Big |\frac{d^k}{dx^k}m_1(x)\Big |\leqq C_k\frac{1}{|x|^k}
$$
with $k\geqq 0$, it follows  from Coifman-Meyer \cite{CM}  and from the fact that 
$\zeta_k$ and $\varphi_k$ and its derivatives have disjoint support, that  for $R_k$ large enough,
\begin{equation}\label{2N_7}
\begin{split}
\int_{\mathbb R} \zeta_k \nu_{n_k}  &J(|D|)\partial_x^2(\varphi_k \nu_{n_k})dx=\frac{1}{\gamma^2} \int_{\mathbb R} \zeta_k \nu_{n_k}  \Big[|D|^2  coth^2(\sqrt{\mu_2}|D|)-\frac{1}{\mu_2}\Big ](\varphi_k \nu_{n_k})dx\\
&=-\frac{1}{\gamma^2} \int_{\mathbb R} \zeta_k \nu_{n_k} \mathcal T_1\partial_x^2(\varphi_k \nu_{n_k})dx=\frac{1}{\gamma^2} \int_{\mathbb R} (\zeta_k \nu_{n_k})' \mathcal T_1(\varphi'_k \nu_{n_k}) + (\zeta_k \nu_{n_k})' \mathcal T_1(\varphi_k \nu'_{n_k})dx\\
&\leqq C \frac{|\varphi '|_{\infty}}{R_k}\|\nu_{n_k}\|_1^2 + \frac{1}{\gamma^2} \int_{\mathbb R} (\zeta_k \nu_{n_k})'[\mathcal T_1, \varphi_k] \nu'_{n_k}dx\\
&\leqq C \frac{|\varphi '|_{\infty}}{R_k}\|\nu_{n_k}\|_1^2 + C\|\nu_{n_k}\|_1\|\varphi'_k\|\|\nu_{n_k}\|\\
&\leqq C \frac{1}{R_k}\|{\bf{h}}_{n_k}\|^2_{1\times 1}=O(\epsilon).
\end{split}
\end{equation}

Thus, for $R_1, R_k$, large enough,  we have from (\ref{key3})
\begin{equation}\label{key7}
E_{\mu_2}({\bf{h}}_{n_k})=E_{\mu_2}({\bf{h}}_{k,1}) + E_{\mu_2}({\bf{h}}_{k,2}) + O(\epsilon).
\end{equation}
Therefore, we obtain that $I_\lambda \geqq I_\theta + I_{\lambda-\theta}$ for some $\theta\in (0, \lambda)$ and so we obtain a contradiction with (\ref{sub2}). Thus, the compactness alternative of Lions's Lemma \ref{lions} will imply the existence of a minimum for the problem (\ref{I}) and then we obtain nontrivial solutions for equation (\ref{0soli}). 

\section{Regularity and asymptotics of the solitary waves}

The spatial decay of solitary waves at infinity is essentially governed by the behavior of the linear part in the solitary wave equation. Paley-Wiener 
type arguments show that no exponential decay is possible when the symbol of the linear operator is not smooth. One expects then an algebraic decay rate, which is the case of the Benjamin-Ono solitary wave (decaying as $1/{x^2})$ or of the KP I equation (see \cite{deBS}).
The situation should be thus quite different for the solitary waves of \eqref{BF1} depending on the value on the finiteness of $\mu_2.$ When $\mu_2$ is finite the nonlocal operator has a smooth symbol and one expects the exponential decay of solitary waves as it is the case for the ILW equation. 

On the other hand when $\mu_2$ is infinite the nonlocal operator involves the Hilbert transform whose symbol has a singularity at the origin and one expects an algebraic decay, as for the BO equation.

Technically following an idea in \cite{BL} one writes the equation for solitary waves as a convolution with the inverse of a linear operator whose symbol determines the decay rate of the solitary waves.

\subsection{Regularity and decay of solitary wave solutions for $\mu_2$ infinite}

We will prove here that the solutions for system (\ref{soli}) have a polynomial decay. More exactly, we have the following theorem

\begin{teo}\label{decay} Let $(\xi, \nu)$ be a solution of system (\ref{soli}) given by Theorem \ref{main}. Then $\xi, \nu \in H^{\infty}(\mathbb R)$. Moreover,  there exists a constant $D\in \mathbb R$ such that for all $x\in \mathbb R$,
$$
x^2 |\xi(x)|\leqq D,\;\;\text{and}\;\; x^2 |\nu(x)|\leqq D.
$$
\end{teo}
\begin{remark}
The decay rate of the solitary wave is exactly the same as that of the (explicit) solitary wave of the Benjamin-Ono equation.
\end{remark}

 \begin{proof} For $(\xi, \nu)$ solution of system (\ref{soli}) we have the following relation for $\xi$
\begin{equation}\label{decay1}
(1-\gamma)\xi= \omega J^{-1}_cJ_b\nu+ rJ^{-1}_c(\nu^2)
\end{equation}
and thus from the first equation in (\ref{soli}), we have the relation for $\nu$
\begin{equation}\label{decay2}
(1-\gamma)\mathcal L_{+\infty}\nu=G(\nu)
\end{equation}
with $G(\nu)=\omega^2 J^{-1}_cJ^2_b\nu + \omega r J^{-1}_cJ_b (\nu^2) + 2\omega r\nu J_c^{- 1}J_b\nu + 2r^2\nu J^{-1}_c(\nu^2)$. We note that by the regularity properties of $\nu$, $G(\nu)\in L^2(\mathbb R)$. Therefore, by using a bootstrap argument we obtain from (\ref{decay2}) that $\nu \in H^{\infty}(\mathbb R)$ and so from the second equation in (\ref{soli}) follows  $\xi \in H^{\infty}(\mathbb R)$.

Next,  equation (\ref{decay2}) can be written in the convolution form
\begin{equation}\label{decay3}
\nu=\frac{1}{1-\gamma} \mathcal L^{-1}_{+\infty}[G(\nu)]=\frac{1}{\beta_1(1-\gamma)}  K\ast G(\nu)
\end{equation}
where $\beta_1= -\frac{\mu}{\gamma}(a-\frac{1}{\gamma^2})$ and 
the kernel $K$ is the inverse Fourier transform of
\begin{equation}\label{decay4}
\widehat{K}(y)=\frac{1}{y^2- \ell |y| +c}
\end{equation}
for $\ell=\frac{\sqrt{\mu}}{\beta_1 \gamma^2}$ and $c=\frac{1}{\beta_1 \gamma}$. So, via the Residue Theorem (see \cite{An}), it  follows that $K$ is determined explicitly by
\begin{equation}\label{decay5}
K(x)= -\frac{2\ell}{\sqrt{2\pi}}\int_0^{+\infty} \frac{y e^{-|x|y}}{(c-y^2)^2 +\ell^2 y^2}dy +
\frac{2 \sqrt{2\pi}}{\sqrt{4c-\ell^2}} e^{-\frac{\sqrt{4c-\ell^2}}{2}|x|}cos(\ell x/2).
\end{equation}
We note that $4c-\ell^2>0$, because $a\leqq 0$ and $1-a\gamma^2>\frac14$. We also have that $K\in C^\infty(\mathbb R-\{0\})$. Now, according to the theory in Bona\&Li \cite{BL} the asymptotic properties of $\nu$ satisfying (\ref{decay3}) are essentially based on those of the kernel $K$. In the present context, since $ \widehat{K}\in H^s(\mathbb R)$ for $s>3/2$, and from (\ref{decay5})
$$
\lim_{|x|\to +\infty} x^2 K(x)=-\frac{2\ell}{c^2\sqrt{2\pi}}.
$$
Then, it follows that $\nu$ satisfies the relation
$$
\lim_{|x|\to \infty} x^2 \nu(x)=D,
$$
where $D$ is a constant. Now, with regard to the polynomial decay of $\xi$ we will see that if $f\in H^n (\mathbb R)$,  satisfies $|x^2 f(x)|\leqq C$ for $x\in \mathbb R$, then $|x^2 J_c^{-1}f(x)|\leqq C_1$ for $x\in \mathbb R$. Indeed, since the Fourier transform of $h(x)=\frac{1}{1-\mu c x^2}$ is given by 
$$
\widehat{h}(y)=\frac{\pi}{\sqrt{-\mu c}} e^{-\sqrt{-\mu c}|y|}
$$
we have for $\theta=\sqrt{-\mu c}$
\begin{equation}\label{decay6}
\begin{split}
\frac{\theta}{\pi} |x^2 J_c^{-1}f(x)|&=|\int_{\mathbb R} e^{-\theta|x-y|}x^2 f(y)dy|\leqq \int_{\mathbb R} e^{-\theta |x-y|}|x-y|^2 |f(y)|dy + \int_{\mathbb R} e^{-\theta |x-y|}y^2 |f(y)|dy\\
&\leqq |f|_{\infty} \int_{\mathbb R} e^{-\theta |p|}p^2 dp + |p^2 f|_{\infty} \int_{\mathbb R} e^{-\theta |p|}dp <\infty.
\end{split}
\end{equation}
Therefore, we obtain that $|x^2 J_c^{-1}(\nu^2)(x)|\leqq D_1$ for all $x\in \mathbb R$. Similarly, by using integration by parts we can obtain $|x^2 J_c^{-1}J_b\nu(x)|\leqq D_2$ for all $x\in \mathbb R$. Thus, from (\ref{decay1}) we have $|x^2 \xi(x)|\leqq D_3$ for all $x\in \mathbb R$. 

This achieves the proof of  the Theorem. 
\end{proof}

\subsection{Regularity and decay of solitary wave solutions for $\mu_2$ finite}

As aforementioned we expect smoothness and exponential decay of the solitary waves solutions here since the symbols of the linear dispersive parts are smooth. We proceed as in the case $\mu_2$ infinite.  More exactly, we have the following theorem.

\begin{teo}\label{decayfi} Let $(\xi, \nu)$ be a solution of system \eqref{0soli} given by Theorem \ref{main}-(b). Then $\xi, \nu\in H^\infty(\mathbb R)$. Moreover, there exists a constant $M\in \mathbb R^+$ such that for all $x\in \mathbb R$,
$$
e^{\sigma|x|}  |\nu(x)|\leqq M ,\;\;\text{and}\;\; e^{\sigma_0|x|} |\xi(x)|\leqq M,
$$
with $\sigma>0$ such that $\sigma^2=-\frac{1}{a\mu}(1+\frac{\mu}{\mu_2\gamma^2}-\frac{\sqrt{\mu}}{\gamma \sqrt{\mu_2}})$ and with $\sigma_0\in (0, \sigma]$ and $\sigma_0<\sqrt{-c\mu}$.
\end{teo}

 \begin{proof} For $(\xi, \nu)$ solution of system \eqref{0soli} we have the following relation for $\xi$
\begin{equation}\label{xi}
(1-\gamma)\xi= \omega J^{-1}_cJ_b\nu+ \frac{\epsilon}{2\gamma}J^{-1}_c(\nu^2)\equiv G_0(\nu).
\end{equation}
Thus, we deduce from \eqref{0soli} that $\nu$ should satisfy the equation
\begin{equation}\label{eqnu} 
(1-\gamma)\mathcal L_{\mu_2}\nu =\frac{\epsilon}{2\gamma} \nu G_0(\nu)+\omega J_bG_0(\nu)\equiv G_1(v).
\end{equation}
Now, since $\nu\in H^1(\mathbb R)$ it follows immediately that $G_1(v)\in L^2(\mathbb R)$. Next we show that $\nu\in H^2(\mathbb R)$. Actually, from the definition for $\mathcal L_{\mu_2}$ in (\ref{L2}) we can split it in two operators $\mathcal M, \mathcal N$, such that $ \gamma\mathcal L_{\mu_2}= \mathcal M+ \mathcal N$ with
\begin{equation}\label{M}
\mathcal M=1-\frac{\sqrt{\mu}}{\gamma \sqrt{\mu_2}}-\frac{\sqrt{\mu}}{\gamma}|D|\Big[\coth(\sqrt{\mu_2}|D|)- \frac{1}{\sqrt{\mu_2} |D|}\Big],
\end{equation}
and 
\begin{equation}\label{N}
 \mathcal N=\frac{\mu}{\gamma^2}\Big[\coth^2(\sqrt{\mu_2}|D|)-\frac{1}{\mu_2 |D|^2}-a\gamma^2\Big]|D|^2 +\frac{\mu}{\mu_2\gamma^2}.
 \end{equation}
Now, from the boundedness of $\mathcal T=\coth(\sqrt{\mu_2}|D|)- \frac{1}{\sqrt{\mu_2} |D|}$ 
on $L^2(\mathbb R)$ it follows from  (\ref{M}) that $\mathcal M \nu\in L^2(\mathbb R)$. By defining $\mathcal T_2=\frac{\mu}{\gamma^2}\Big [\coth^2(\sqrt{\mu_2}|D|)-\frac{1}{\mu_2 |D|^2}-a\gamma^2\Big]|D|^2$ we obtain from (\ref{eqnu}) that
 \begin{equation}\label{N2}
\mathcal T_2\nu= \frac{\gamma}{1-\gamma} G_1(\nu)-\mathcal M\nu-\frac{\mu}{\mu_2\gamma^2} \nu\in L^2(\mathbb R).
 \end{equation}
Thus, since $\coth^2(\sqrt{\mu_2}|x|)-\frac{1}{\mu_2 |x|^2}-a\gamma^2\geqq  -a\gamma^2$, we obtain from $a<0$ the relation $\|\mathcal T_2\nu\|^2\geqq \mu^2 a^2\int |y|^4|\hat{v}(y)|^2dy$ and therefore $\nu \in H^2(\mathbb R)$. Thus, a bootstrap argument applied to equation (\ref{eqnu}) implies $\nu\in H^\infty(\mathbb R)$ and from (\ref{decay1}) follows immediately  that $\xi\in H^\infty(\mathbb R)$.

Next, we prove the exponential decay of the profile $\nu$. Indeed, from (\ref{eqnu})  follows the relation
\begin{equation}\label{N3}
-a\mu(|D|^2 +\sigma^2)\nu= \frac{\gamma}{1-\gamma} G_1(\nu)+ \frac{\sqrt{\mu}}{\gamma}|D| \mathcal T\nu-\frac{\mu}{\gamma^2}|D|^2 \mathcal T_1v\equiv G_2(\nu)\in L^2(\mathbb R),
 \end{equation}
where $\mathcal T_1$ was defined in (\ref{T_1}) and $\sigma>0$ is such that $\sigma^2=-\frac{1}{a\mu}(1+\frac{\mu}{\mu_2\gamma^2}-\frac{\sqrt{\mu}}{\gamma \sqrt{\mu_2}})$ (we note from Theorem \ref{main}-(b) that $1>\frac{\sqrt{\mu}}{\gamma \sqrt{\mu_2}}$). Thus, from (\ref{N3}) we obtain
\begin{equation}\label{N4}
\nu=-\frac{1}{a\sigma \mu}K_1\star G_2(\nu),
\end{equation}
where $K_1$ is the inverse Fourier transform of
$$
\widehat{K_1}(y)=\frac{\sigma}{\sigma^2+y^2}.
$$
Namely, $K_1(x)=\pi e^{-\sigma |x|}$. Thus, from Bona\&Li in \cite{BL} follows the relation 
\begin{equation}\label{N5}
\lim_{|x|\to \infty} e^{\sigma|x|} \nu(x)=C,
\end{equation}
where $C$ is a constant.

Now, with regard to the exponential decay of $\xi$ in (\ref{xi}) we have  that if $g\in H^n (\mathbb R)$,  satisfies $| e^{\sigma |x|}g(x)|\leqq C_1$ for $x\in \mathbb R$, then $|e^{\sigma _0|x|}J_c^{-1}g(x)|\leqq C_2$ for $x\in \mathbb R$ with $\sigma_0\in (0, \sigma]$ and $\sigma_0<\sqrt{-c\mu}$. Indeed, for $\theta=\sqrt{-\mu c}$ we have
\begin{equation}\label{decay7}
\begin{split}
\frac{\theta}{\pi} |e^{\sigma _0|x|}J_c^{-1}g(x)|&\leqq \int_{\mathbb R} e^{(\sigma_0-\theta) |x-y|} e^{\sigma _0|y|}|g(y)|dy\\
&\leqq sup_{y\in \mathbb R}|e^{\sigma _0|y|} g(y)| \int_{\mathbb R} e^{(\sigma_0-\theta) |p|} dp  <\infty.
\end{split}
\end{equation}
Therefore, we obtain that $|e^{\sigma _0|x|}J_c^{-1}(\nu^2)(x)|\leqq D_1$ and $|e^{\sigma _0|x|}J_c^{-1}(J_b\nu)(x)|\leqq D_2$ for all $x\in \mathbb R$ (we note from (\ref{N4}) that $-a\sigma \mu J_b\nu=K_1\star  (J_b G_2(\nu))$ and so $| e^{\sigma |x|}J_bv(x)|\leqq C_4$ for $x\in \mathbb R$). Thus, from (\ref{xi}) we have $| e^{\sigma _0|x|}\xi(x)|\leqq D_3$ for all $x\in \mathbb R$. 

This achieves the proof of  the Theorem. 
\end{proof}

\section{Existence of solitary waves solutions for BO systems}

In this section we show the existence of even solitary waves solutions for the Benjamin-Ono system (henceforth BO-systems)
%\begin{equation}\label{BO1}
%\left\{\begin{array}{lll}
%\mathcal D\partial_t\zeta +\mathcal B\partial_x v -\frac{\epsilon}{\gamma}\partial_x(\zeta v)=0\\
%\partial_t v+ (1-\gamma) \partial_x \zeta -\frac{\epsilon}{2\gamma}\partial_x(v^2)=0,
%\end{array}\right.
%\end{equation}
%where $\mathcal D$ and $\mathcal B$ are self-adjoint operators defined for $\gamma\in (0,1)$, $\beta> 1$??? and $\mu, \varepsilon>0$ by
%\begin{equation}\label{BO2}
%\mathcal D=1+\frac{\beta}{\gamma} \sqrt{\mu} |D|,\qquad \mathcal B=\frac{1}{\gamma}\Big(1+\frac{\beta-1}{\gamma} \sqrt{\mu} |D|\Big ),
%\end{equation}
%with $|D|=\mathcal H\partial_x$, where $\mathcal H$ represents the Hilbert transform.

Solitary waves solutions for (\ref{BO1}), that is of the form
$$
\zeta(x,t)=\xi(x-ct),\quad v(x,t)=\nu (x-ct),\qquad c\in \mathbb R
$$
and $(\xi, \nu)$ vanish at infinity, it will satisfy the systems
\begin{equation}\label{BO3}
\left\{\begin{array}{lll}
-c\mathcal D \xi +\mathcal B\nu =\frac{\epsilon}{\gamma} \xi\nu,\\
\xi=\frac{1}{1-\gamma} \Big (c \nu + \frac{\epsilon}{2\gamma}\nu^2\Big).
\end{array}\right.
\end{equation}
Our approach of the existence of a smooth curve $c\in (-\delta, \delta)\to (\xi_c, \nu_c)$ of solutions for (\ref{BO3}) will be based on the Implicit Function Theorem. 

Indeed, for $s\geqq 0$, let $H^s_e(\mathbb R)$ denote the closed subspace of all even functions in $H^s(\mathbb R)$. Our existence theorem is the following,

\begin{teo}\label{exisBO}
Let $\gamma\in (0,1)$. Then there exists $\delta>0$ such that for $c\in (-\delta, \delta)$, equation (\ref{BO3}) has a solution $(\xi_c, \nu_c)\in H^1_e(\mathbb R)\times H^1_e(\mathbb R)$, and the map $c\in (-\delta, \delta)\to (\xi_c, \nu_c)$ is smooth. In particular, $(\xi_c(x), \nu_c(x))$ converges to $(\xi_0(x), \nu_0(x))$ as $c\to 0$, uniformly for $x\in \mathbb R$, where $\nu_0$ is defined  by the unique (modulo translations) even and positive solution of 
\begin{equation}\label{BO4}
\alpha D\nu_0+ \frac{1}{\gamma} \nu_0 -\eta \nu_0^3=0,
\end{equation}
and $\xi_0=\frac{1}{1-\gamma} \frac{\epsilon}{2\gamma}\nu_0^2$. Here, $\alpha=\frac{\beta-1}{\gamma^2} \sqrt{\mu}$ and $\eta =\frac{\epsilon^2}{2\gamma^2(1-\gamma)}$.
\end{teo}

\begin{proof} Let $X_e=H^1_e(\mathbb R)\times H^1_e(\mathbb R)$ and define a map $G:\mathbb R\times X_e\to L^2_e(\mathbb R)\times L^2_e(\mathbb R)$ by
$$
G(c, \xi, \nu)= (-c\mathcal D \xi +\mathcal B\nu-\frac{\epsilon}{\gamma} \xi\nu, -c\nu +(1-\gamma)\xi- \frac{\epsilon}{2\gamma}\nu^2).
$$
Then, from \cite{FL}, we obtain that there is a unique  even  solution $(\xi_0, \nu_0)$ of  equation $G(0, \xi_0, \nu_0)=0$, with $\nu_0$ satisfying (\ref{BO4}) and $\xi_0, \nu_0$ positive. Next, a calculation shows that the Fr\'echet derivative $G_{(\xi,\nu)}= \partial G(c, \xi, \nu)/{\partial (\xi, \nu)}$ exists on $\mathbb R\times X_e$ and is defined as a map from $\mathbb R\times X_e$ to $B(X_e; L^2_e(\mathbb R)\times L^2_e(\mathbb R))$ by
$$
G_{(\xi,\nu)}(c, \xi,\nu)=\left(\begin{array}{cc}
-c\mathcal D- \frac{\epsilon}{\gamma}\nu & \mathcal B-\frac{\epsilon}{\gamma}\xi \\
1-\gamma & -\frac{\epsilon}{\gamma}\nu-c
\end{array}\right)
$$
Moreover the linear operator $\mathcal T_0=G_{(\xi,\nu)}(0, \xi_0,\nu_0)$ with domain $D(\mathcal  T_0)= H^1(\mathbb R)\times H^1(\mathbb R)$ has a one-dimensional kernel, $Ker(\mathcal T_0)$, generated by $( \xi'_0,\nu'_0)^t$. Indeed, by considering $c=0$ in (\ref{BO3}) we have that $( \xi'_0,\nu'_0)$ satisfies
$$
\mathcal Bv'_0=\frac{\epsilon}{\gamma}(\xi_0'\nu_0+\xi_0\nu_0'),\;\; \text{and},\;\;\xi_0'=\frac{1}{1-\gamma} \frac{\epsilon}{\gamma}\nu_0\nu_0',
$$
and hence $( \xi'_0,\nu'_0)^t\in Ker(\mathcal T_0)$.

 Next, suppose $(\phi, \psi)^t\in Ker(\mathcal T_0)$. Then, since $\phi$ satisfies
\begin{equation}\label{BO5}
\phi=\frac{\epsilon}{\gamma(1-\gamma)} \nu_0\psi, 
\end{equation}
we deduce that $\psi$ belongs to the kernel of the linear operator 
$$
\mathcal M_0= \alpha D+ \frac{1}{\gamma}  -3\eta \nu_0^2.
$$
Thus, from \cite{FL} we have that $Ker(\mathcal M_0)=[\nu_0']$ and the number of negative eigenvalues of $\mathcal M_0$ is exactly one. Then, $\psi=\theta \nu'_0$ and from (\ref{BO5}) $\phi=\theta \xi_0$.

Now, since  $( \xi'_0,\nu'_0)\notin X_e$ we obtain that $\mathcal T_0:X_e\to L^2_e(\mathbb R)\times L^2_e(\mathbb R)$ is invertible. Moreover, since $G$ and $G_{(\xi,\nu)}$ are smooth maps on their domains, we have from the Implicit Function Theorem that there exist a number $\delta>0$ and a smooth map $c\in (-\delta, \delta)\to (\xi_c, \nu_c)\in X_e$ such that $G(c, \xi_c, \nu_c)=0$ for all $c\in (-\delta, \delta)$. This proves the theorem.

\end{proof}

Theorem \ref{exisBO} shows the existence of a smooth curve  of solutions for (\ref{BO3}) bifurcating from the ``positive'' profile  $(\xi_0, \nu_0)$ with $\nu_0$ being the positive solution for (\ref{BO4}). Since $-\nu_0$ is the negative solution associated to (\ref{BO4}), a similar analysis to that used in the proof of Theorem \ref{exisBO} shows the following existence result of solitary waves solutions for (\ref{BO3}).

\begin{teo}\label{exisBO2}
Let $\gamma\in (0,1)$. Then there exists $\delta_1>0$ such that for $c\in (-\delta_1, \delta_1)$, equation (\ref{BO3}) has a solution $(\xi_{c, 1}, \nu_{c,1})\in H^1_e(\mathbb R)\times H^1_e(\mathbb R)$, and the correspondence $c\in (-\delta_1, \delta_1)\to (\xi_{c, 1}, \nu_{c,1})$ define a smooth map. In particular, for $c\to 0$, $(\xi_{c,1}(x), \nu_{c, 1}(x))$ converges to $(\xi_0(x), -\nu_0(x))$, uniformly for $x\in \mathbb R$, where $\nu_0$ is  the unique (modulo translations) even and positive solution of (\ref{BO4})
and 
$$
\xi_0=\frac{1}{1-\gamma} \frac{\epsilon}{2\gamma}\nu_0^2.
$$ 
Moreover, from Theorem \ref{exisBO2}  we obtain that $\delta_1=\delta$ and 
$$
 (\xi_{c, 1}, \nu_{c,1})= (\xi_{c}, -\nu_{c}), \quad \text{for}\;\;c\in (-\delta, \delta).
$$
\end{teo}

We note that Theorems \ref{exisBO} and \ref{exisBO2} show the bidirectional nature of the system (\ref{BO1}).

\section{Existence of solitary waves solutions for ILW systems}

In this section we show the existence of even solitary waves solutions for the Intermediate Long Wave systems (henceforth ILW-systems)
%\begin{equation}\label{W1}
%\left\{\begin{array}{lll}
%\mathcal W\partial_t\zeta +\mathcal Z\partial_x v -\frac{\epsilon}{\gamma}\partial_x(\zeta v)=0\\
%\partial_t v+ (1-\gamma) \partial_x \zeta -\frac{\epsilon}{2\gamma}\partial_x(v^2)=0,
%\end{array}\right.
%\end{equation}
%where $\mathcal W$ and $\mathcal Z$ are self-adjoint operators defined for $\gamma\in (0,1)$, $\beta> 1$???, $\mu, \mu_2 \varepsilon>0$ by
%\begin{equation}\label{W2}
%\mathcal W=1+g(D),\qquad \mathcal Z=\frac{1}{\gamma}\Big(1+\frac{\beta-1}{\gamma} \sqrt{\mu} |D| coth(\sqrt{\mu_2}|D|)\Big ),
%\end{equation}
%with $|D|=\mathcal H\partial_x$, where $\mathcal H$ represents the Hilbert transform, and 
%$$
%g(D)=\frac{\beta}{\gamma}\sqrt{\mu} |D| coth(\sqrt{\mu_2}|D|).
%$$

Solitary waves solutions for (\ref{W1}), that is of the form
$$
\zeta(x,t)=\xi(x-ct),\quad v(x,t)=\nu (x-ct),\qquad c\in \mathbb R
$$
and $(\xi, \nu)$ vanish at infinity, it will satisfy the systems
\begin{equation}\label{W3}
\left\{\begin{array}{lll}
-c\mathcal W \xi +\mathcal Z\nu =\frac{\epsilon}{\gamma} \xi\nu,\\
\xi=\frac{1}{1-\gamma} \Big (c \nu + \frac{\epsilon}{2\gamma}\nu^2\Big).
\end{array}\right.
\end{equation}
Our approach of the existence of a smooth curve $c\in (-\delta, \delta)\to (\xi_c, \nu_c)$ of solutions for (\ref{BO3}) will be based again in the Implicit Function Theorem. First, we prove  that for $c=0$, system (\ref{W3}) has a smooth curve $\mu_2\to (\xi_{0, \mu_2},\nu_{0, \mu_2})\in H^1_e(\mathbb R)\times H^1_e(\mathbb R)$ with $\mu_2$ sufficiently large.

\begin{teo}\label{exisW}
Let $\gamma\in (0,1)$ and $c=0$ in (\ref{W3}). Then there exists $\sigma>0$ sufficiently large
such that for $\mu_2\in (\sigma, +\infty)$, equation (\ref{W3}) has a solution $(\xi_{0, \mu_2}, \nu_{0, \mu_2})\in H^1_e(\mathbb R)\times H^1_e(\mathbb R)$, and the correspondence $\mu_2\in (\sigma, +\infty)\to (\xi_{0, \mu_2}, \nu_{0, \mu_2})$ define a smooth map. In particular, for $\mu_2\to +\infty$, $(\xi_{0, \mu_2}(x), \nu_{0, \mu_2}(x))$ converges to $(\xi_0(x), \nu_0(x))$, uniformly for $x\in \mathbb R$, where $\xi_0, \nu_0$ are defined in Theorem \ref{exisBO}.

\end{teo}

\begin{proof} Let $X_e=H^1_e(\mathbb R)\times H^1_e(\mathbb R)$ and define a map $H:\mathbb R\times X_e\to L^2_e(\mathbb R)\times L^2_e(\mathbb R)$ by
$$
P(a, \xi, \nu)= (\frac1\gamma \nu -\frac{\epsilon}{\gamma} \xi\nu +\alpha|a|\tau_{|a|}[|D|coth(|D|)\tau_{\frac{1}{|a|}}\nu],  (1-\gamma)\xi- \frac{\epsilon}{2\gamma}\nu^2),
$$
where $\tau_d$ is the linear dilation operator, $\tau_df(x)\equiv f(dx)$. For $d=0$, $\tau_{\frac1d}f\equiv 0$. We note from the relation, 
\begin{equation}\label{inequa}
0\leqq xcoth\Big(\frac{1}{d}x\Big)-x\leqq d,\;\;\qquad{for\; all}\;\;x\geqq 0\;\;\text{and}\;\; d\geqq 0
\end{equation}
that $|D|coth(|D|)v \in L^2(\mathbb R)$ for $v\in H^1(\mathbb R)$ and
$$
|a|\tau_{|a|}[|D|coth(|D|)\tau_{\frac{1}{|a|}}\nu]\to |D|\nu,\quad\text{as}\;\;a\to 0
$$
in $L^2(\mathbb R)$, because of the relation
\begin{equation}\label{W4}
|a|\tau_{|a|}[|D|coth(|D|)\tau_{\frac{1}{|a|}}\nu]=|D|coth\Big(\frac{1}{|a|}|D|\Big)\nu
\end{equation}
Then from \cite{FL}, we obtain that there is a unique  even  solution $(\xi_0, \nu_0)$ of  equation 
$$
(0,0)=P(0, \xi_0, \nu_0)=(\frac1\gamma \nu_0 -\frac{\epsilon}{\gamma} \xi_0\nu_0 +\alpha |D|\nu_0,  (1-\gamma)\xi_0- \frac{\epsilon}{2\gamma}\nu_0^2),
$$ 
with $\nu_0$ satisfying (\ref{BO4}) and $\xi_0, \nu_0$ positive. 

Next, we calculate the Fr\'echet derivative $\mathcal T(a,\xi, \nu)=\partial P(a, \xi, \nu)/{\partial (\xi, \nu)}$ which is a map from $\mathbb R\times X_e$ to $B(X_e; L^2_e(\mathbb R)\times L^2_e(\mathbb R))$
$$
\mathcal T(a, \xi, \nu)=\left(\begin{array}{cc}
- \frac{\epsilon}{\gamma}\nu & \frac{1}{\gamma}-\frac{\epsilon}{\gamma}\xi +\alpha |a| \tau_{|a|}[|D|coth(|D|)\tau_{\frac{1}{|a|}}] \\
1-\gamma & -\frac{\epsilon}{\gamma}\nu
\end{array}\right)
$$
and so the linear operator $\mathcal T_0=\mathcal T(0, \xi_0,\nu_0)$ defined with domain $H_e^1(\mathbb R)\times H_e^1(\mathbb R)$ has a trivial kernel (see proof of Theorem \ref{exisBO2}). Therefore,  $\mathcal T_0:X_e\to L^2_e(\mathbb R)\times L^2_e(\mathbb R)$ is invertible. Moreover, since $P$ and $\mathcal T$ are continuous maps on their domains, we have from the Implicit Function Theorem that there exist a number $\sigma_1>0$ and a continuous map $a\in (-\sigma_1, \sigma_1)\to (\xi_{0, a}, \nu_{0,a})\in X_e$ such that $P(a, \xi_{0, a}, \nu_{0, a})=0$ for all $a\in (-\sigma_1, \sigma_1)$. We note that since $P$ and $\mathcal T$ are smooth maps on the domain $(0, +\infty)$, then $a\in (0,\sigma_1)\to (\xi_{0, a}, \nu_{0,a})\in X_e$ is also a smooth map.

Now, from relation (\ref{W4}) we obtain for $a=\frac{1}{\sqrt{\mu_2}}$ and $\mu_2>\sigma$ that $(\xi_{0, a}, \nu_{0,a})\equiv (\xi_{0, \mu_2}, \nu_{0,\mu_2})$ satisfies (\ref{W3}) with $c=0$. This shows the theorem.

\end{proof}

Next, we show the existence of a smooth curve of solution for (\ref{W3}) depending of the velocity $c$.

\begin{teo}\label{exisW1}
Let $\gamma\in (0,1)$ and $\mu_2$ sufficiently large. Then there exists $\eta>0$ such that for $c\in (-\eta, \eta)$, equation (\ref{W3}) has a solution $(\xi_{c, \mu_2}, \nu_{c, \mu_2})\in H^1_e(\mathbb R)\times H^1_e(\mathbb R)$, and the correspondence $c\in (-\eta, \eta)\to (\xi_{c, \mu_2}\nu_{c, \mu_2})$ define a smooth map. In particular, for $c\to 0$, $(\xi_{c, \mu_2}(x), \nu_{c, \mu_2}(x))$ converges to $(\xi_{0, \mu_2}(x), \nu_{0, \mu_2}(x))$, uniformly for $x\in \mathbb R$, where $(\xi_{0,\mu_2},\nu_{0,\mu_2})$ is  the solution of (\ref{W3}) with $c=0$ and defined by Theorem \ref{exisW}. 
\end{teo}

\begin{proof} Let $X_e=H^1_e(\mathbb R)\times H^1_e(\mathbb R)$ and define a map $H:\mathbb R\times X_e\to L^2_e(\mathbb R)\times L^2_e(\mathbb R)$ by
$$
Q(c, \xi, \nu)= (-c(1+g(D)) \xi+\frac1\gamma \nu -\frac{\epsilon}{\gamma} \xi\nu +\alpha|D|coth(\sqrt{\mu_2}|D|)\nu,  (1-\gamma)\xi- c\nu-\frac{\epsilon}{2\gamma}\nu^2).
$$
From Theorem \ref{exisW} we obtain $Q(0,  \xi_{0,\mu_2}, \nu_{0,\mu_2})=0$. Next, the Fr\'echet derivative $\mathcal Z(c,\xi, \nu)=\partial Q(c, \xi, \nu)/{\partial (\xi, \nu)}$ which is a map from $\mathbb R\times X_e$ to $B(X_e; L^2_e(\mathbb R)\times L^2_e(\mathbb R))$ is given by
$$
\mathcal Z(c, \xi, \nu)=\left(\begin{array}{cc}
-c(1+g(D))- \frac{\epsilon}{\gamma}\nu & \frac{1}{\gamma}-\frac{\epsilon}{\gamma}\xi +\alpha |D|coth(\sqrt{\mu_2}|D|)\\
1-\gamma & -c -\frac{\epsilon}{\gamma}\nu
\end{array}\right)
$$

Next, we will see that $\mathcal Z_{\mu_2}\equiv\mathcal Z(0, \xi_{0,\mu_2}, \nu_{0,\mu_2})$ has a trivial kernel on $X_e$ for $\mu_2$ sufficiently large. Indeed, initially we have that on $H^1(\mathbb R)\times H^1(\mathbb R)$, $( \xi'_{0,\mu_2}, \nu'_{0,\mu_2})\in Ker (\mathcal Z_{\mu_2})$ for all $\mu_2$. Now, we show that $\mathcal Z_{\mu_2}$ converges to $\mathcal T_0$, as $\mu_2\to+\infty$,  in the topology of generalized convergence on the set $\mathcal C$ of all closed operators on $L^2(\mathbb R)\times L^2(\mathbb R)$, namely, for $\widehat{\delta}(S, T)$ denoting a {\it metric gap} between the closed operators $S$ and $T$ with $D(S), D(T)\subset L^2(\mathbb R)\times L^2(\mathbb R)$ (see the Appendix), we have
\begin{equation}\label{gap1}
\lim_{\mu_2\to +\infty} \widehat{\delta}(\mathcal Z_{\mu_2}, \mathcal T_0)=0.
\end{equation}
Indeed, for
$$
\mathcal V_{\mu_2}= \left(\begin{array}{cc}
- \frac{\epsilon}{\gamma}\nu_{0, \mu_2} & -\frac{\epsilon}{\gamma}\xi_{0, \mu_2}\\
1-\gamma & -\frac{\epsilon}{\gamma}\nu_{0, \mu_2}
\end{array}\right),\qquad \mathcal V_{+\infty}= \left(\begin{array}{cc}
- \frac{\epsilon}{\gamma}\nu_{0} & -\frac{\epsilon}{\gamma}\xi_{0}\\
1-\gamma & -\frac{\epsilon}{\gamma}\nu_{0}
\end{array}\right)
$$
and 
$$
\mathcal L_{\mu_2}=\left(\begin{array}{cc}
0 & \alpha |D|coth(\sqrt{\mu_2}|D|)-\alpha |D|\\
0& 0
\end{array}\right), \qquad \mathcal S_{+\infty}=\left(\begin{array}{cc}
0 & \mathcal B\\
0& 0
\end{array}\right)
$$
with $\mathcal B$ defined in (\ref{BO2}), we have $\mathcal Z_{\mu_2}= \mathcal L_{\mu_2} + \mathcal S_{+\infty}+\mathcal V_{\mu_2}$ and $\mathcal T_0=\mathcal S_{+\infty}+ \mathcal V_{+\infty}$. Then, from the relations (a)-(b) in Theorem \ref{A} we have
\begin{equation}\label{gap2}
\begin{array}{lll}
\widehat{\delta}(\mathcal Z_{\mu_2}, \mathcal T_0)&=\widehat{\delta}(\mathcal L_{\mu_2} + \mathcal S_{+\infty}+\mathcal V_{\mu_2}, \mathcal S_{+\infty}+ (\mathcal V_{+\infty}-\mathcal V_{\mu_2})+\mathcal V_{\mu_2})\\
\\
&\leqq  2(1+\|\mathcal V_{\mu_2}\|^2_{B(L^2)}) \widehat{\delta}(\mathcal L_{\mu_2} + \mathcal S_{+\infty}, \mathcal S_{+\infty}+ (\mathcal V_{+\infty}-\mathcal V_{\mu_2}))
\\
\\
&\leqq 2(1+\|\mathcal V_{\mu_2}\|^2_{B(L^2)})[\widehat{\delta}(\mathcal L_{\mu_2} + \mathcal S_{+\infty}, \mathcal S_{+\infty})+ \widehat{\delta}(\mathcal S_{+\infty},  \mathcal S_{+\infty}+(\mathcal V_{+\infty}-\mathcal V_{\mu_2}))]
\\
\\
&\leqq 2(1+\|\mathcal V_{\mu_2}\|^2_{B(L^2)})[\|\mathcal L_{\mu_2}\|_{B(L^2)} + \|\mathcal V_{\mu_2}-\mathcal V_{+\infty}\|_{B(L^2)}] \to 0
\end{array}
\end{equation}
as $\mu_2\to +\infty$, where we are used Theorem \ref{exisW} for obtaining 
$$
\lim_{\mu_2\to +\infty}\|\mathcal V_{\mu_2}-\mathcal V_{+\infty}\|_{B(L^2)}=0,
$$
and inequality (\ref{inequa}) for obtaining via Plancherel's Theorem that
\begin{equation}
\begin{array}{lll}
\|\mathcal L_{\mu_2}(f,g)^t\|^2&=\alpha \int_{\mathbb R} | |\xi| \coth(\sqrt{\mu_2}|\xi|)-|\xi||^2|\widehat{g}(\xi)|^2d\xi\\
\\
&\leqq \frac{\alpha}{\mu_2}\|g\|^2\to 0
\end{array}
\end{equation}
as $\mu_2\to +\infty$.

Hence the results of section IV-4 of Kato \cite{K} imply that the eigenvalues of $\mathcal Z_{\mu_2}$ depend continuously on $\mu_2$. In particular,  since zero is a simple  eigenvalue of $\mathcal T_0$, the eigenvalue zero with eigenfunction $( \xi'_{0,\mu_2}, \nu'_{0,\mu_2})$ is simple for $\mathcal Z_{\mu_2}$ and all $\mu_2$ sufficiently large. Also, for these values of $\mu_2$, zero is not a eigenvalue of $\mathcal Z_{\mu_2}$ in $X_e$, and therefore the mapping $\mathcal Z_{\mu_2}: X_e\to L^2(\mathbb R)\times L^2(\mathbb R)$ is invertible. Moreover, since $Q$ and $\mathcal Z_{(c,\xi,\nu)}$ are smooth maps on their domains, we have from the Implicit Function Theorem that there exist a number $\eta>0$ and a smooth map $c\in (-\eta, \eta)\to (\xi_{c, \mu_2}, \nu_{c, \mu_2})\in X_e$ such that $Q(c, \xi_{c, \mu_2}, \nu_{c, \mu_2})=0$ for all $c\in (-\eta, \eta)$ and $\mu_2$ sufficiently large. This proves the theorem.

\end{proof}

\section{Properties of the ILW and BO systems solitary waves}
Similarly to the case of Boussinesq-Full dispersion systems we establish here qualitative  properties of the solitary wave solutions to the ILW and BO systems.

 Concerning the smoothness of the solitary wave solutions of the ILW and BO systems, one cannot apparently implement a bootstrap argument from \eqref{BO3} or \eqref{W3}. Since the solitary wave solutions to both the BO and ILW equation are smooth ($H^\infty(\R)),$  one has instead to apply the implicit function theorem in the Sobolev space $H^n (\R)$ where  n is arbitrary large.

On the other hand, the proof of  the decay properties of the solitary waves follows the lines of the proof of the corresponding properties for the Boussinesq-Full dispersion  solitary waves (see Section 4) yielding  that the solitary waves of the ILW system decay exponentially while those of the BO system have a algebraic decay $1/x^2$. We give some details below.

\subsection{Decay of solitary wave solutions for the BO system}
We will prove that the solutions for BO system (\ref{BO3}) have a polynomial decay. More exactly, we have the following result.

\begin{teo}\label{BO} Let $(\xi, \nu)$ be a solution of system (\ref{BO3}) given by Theorem \ref{exisBO} (or by Theorem \ref{exisBO2}). Then there exists a constant $N\in \mathbb R$ such that for all $x\in \mathbb R$,
$$
x^2 |\xi(x)|\leqq N,\;\;\text{and}\;\; x^2 |\nu(x)|\leqq N.
$$
\end{teo}

 \begin{proof} If $(\xi, \nu)$ is a solution of system (\ref{BO3}) we have the following relation for $\nu$
\begin{equation}\label{decay8}
(\alpha + |D|)\nu=\gamma \alpha (P(\nu)+c H(\nu))\equiv G_3(\nu)\in L^2(\mathbb R)
\end{equation}
where $\alpha=\frac{\gamma}{(\beta-1)\sqrt{\mu}}>0$, $P(\nu)=\frac{\epsilon}{\gamma(1-\gamma)} \Big (c \nu + \frac{\epsilon}{2\gamma}\nu^2\Big)$ and $H(\nu)= \frac{1}{1-\gamma} \mathcal D \Big (c \nu + \frac{\epsilon}{2\gamma}\nu^2\Big)$.
Hence, equation (\ref{decay8}) can be written in  convolution form
\begin{equation}\label{decay9}
\nu=K_2\ast G_3(\nu)
\end{equation}
where the kernel $K_2$ is the inverse Fourier transform of
\begin{equation}\label{decay10}
\widehat{K_2}(y)=\frac{1}{|y| +\alpha}
\end{equation}
So, via the Residue Theorem, it  follows that $K_2$ is determined explicitly by
\begin{equation}\label{decay11}
K_2(x)= \frac{\sqrt{2}}{\sqrt{\pi}}\int_0^{+\infty} \frac{y e^{-|x|y}}{\alpha^2 + y^2}dy,
\end{equation}
and consequently
$$
\lim_{|x|\to +\infty} x^2 K_2(x)= \frac{\sqrt{2}}{\sqrt{\pi}}\frac{1}{\alpha^2}.
$$
Then, it follows that $\nu$ satisfies the relation
$$
\lim_{|x|\to \infty} x^2 \nu(x)=D_1,
$$
where $D_1$ is a constant. Moreover, from the second equation in (\ref{BO3}) follows immediately that $ x^2 |\xi(x)|\leqq N$ for all $x\in \mathbb R$. This finishes the theorem.
\end{proof}

\subsection{Decay of solitary wave solutions for the ILW system}
We will prove that the solutions for ILW system (\ref{W3}) have a exponential decay. More exactly, we have the following result.

\begin{teo}\label{W} Let $(\xi, \nu)$ be a solution of system (\ref{W3}) given by Theorem \ref{exisW1}. Then there are positive constants $\eta_1, N_1>0$ such that for all $x\in \mathbb R$,
$$
e^{\eta_1|x|} |\xi(x)|\leqq N_1,\;\;\text{and}\;\; e^{\eta_1|x|}  |\nu(x)|\leqq N.
$$
\end{teo}

 \begin{proof} For $(\xi, \nu) \in H^1(\mathbb R)\times H^1(\mathbb R)$ solution of system (\ref{W3}) we will see that $\mathcal W \xi\in L^2(\mathbb R) $. Indeed, from the relation
 $$
\mathcal W \xi=\Big(1+\frac{\beta\sqrt{\mu}}{\gamma \sqrt{\mu_2}}\Big)\xi+ \frac{\beta\sqrt{\mu}}{\gamma}\mathcal T|D|\xi
 $$
 where $\mathcal T=coth(\sqrt{\mu_2}|D|)-\frac{1}{\sqrt{\mu_2}|D|}$ is a bounded operator on $L^2(\mathbb R)$, we obtain immediately that $\mathcal W \xi\in L^2(\mathbb R)$. Moreover, from (\ref{W3}) follows that $\gamma \mathcal Z \nu=G_4(\nu)\in  L^2(\mathbb R)$. Thus, we obtain the 
relation
\begin{equation}\label{decay12}
\nu=K_3\ast G_4(\nu)
\end{equation}
where the kernel $K_3$ is the inverse Fourier transform of
\begin{equation}\label{decay13}
\widehat{K_3}(y)=\frac{1}{1+\frac{\beta-1}{\gamma}\sqrt{\mu}|y| coth(\sqrt{\mu_2} |y|)}=\theta\frac{sinh(\sqrt{\mu_2} y)}{\sqrt{\mu_2} y cosh(\sqrt{\mu_2} y)+ \theta sinh(\sqrt{\mu_2} y)}.
\end{equation}
where $\theta=\frac{\gamma\sqrt{\mu_2}}{(\beta-1)\sqrt{\mu}}>0$. Now, in order to find  $K_3$ we define the following function $h_\theta$,
\begin{equation}\label{decay14}
\widehat{h_\theta}(y)=\frac{sinh(y)}{ y cosh( y)+ \theta sinh( y)}.
\end{equation}
Then, the meromorphic function on the right hand-side of (\ref{decay14})  has  countably many poles located at point $y=i\eta$ where $\eta\in \mathbb R$ and $\eta \neq 0$ satisfy the relation $-\eta=\theta tan\; \eta$.  Thus an application of the residue theorem implies that its inverse Fourier transform can be expressed as 
\begin{equation}\label{decay15}
h_\theta(x)=\sqrt{2\pi} \sum_{m=1}^\infty\frac{tan\; \eta_m}{\eta_m tan\; \eta_m-\theta-1}e^{-\eta_m |x|}
\end{equation}
where $\{i\eta_m\}_{m\geqq 1}$ comprise the poles of $\widehat{h_\theta}$ on the positive  imaginary axis, numbered so that $\eta_m<\eta_{m+1}$ for $m=1, 2, 3, \cdot\cdot\cdot$. More exactly, $\eta_m\in (\frac{(2m-1)}{2}\pi, m\pi)$.

Now, from the relation $K_3(x)=\theta \frac{1}{\sqrt{\mu_2}} h_\theta (\frac{x}{\sqrt{\mu_2}})$ we obtain from (\ref{decay15})  for any $\sigma$ with $0<\sigma \leqq \eta_1$ that
$$
\lim_{|x|\to +\infty} e^{\sigma |x|} K_3(x)= \frac{\sqrt{2\pi}\gamma}{(\beta-1) \sqrt{\mu}}\;\frac{tan\; \eta_1}{\eta_1 tan \;\eta_1 -\theta-1}
$$
Then, from (\ref{decay12}) follows that $\nu$ satisfies the relation
$$
\lim_{|x|\to \infty} e^{\eta_1 |x|}  \nu(x)=D_2,
$$
where $D_2$ is a constant. Moreover, from the second equation in (\ref{W3}) follows immediately $e^{\eta_1 |x|}   |\xi(x)|\leqq N_1$ for all $x\in \mathbb R$. This finishes the theorem.
\end{proof}

\section{Conclusion and open problems}
We have establish the existence of solitary wave solutions for the Boussinesq-Full Dispersion, Intermediate Long Wave and Benjamin-Ono systems for a rather restricted range of parameters. We indicate below some related open questions that will be addressed in subsequent works.
\begin{itemize}
\item Complete the Cauchy theory for the B-FD systems, in particular long time existence issues.
\item Existence in other ranges of parameters for the existence of solitary waves for the B-FD, BO and ILW systems.
\item Non existence of solitary waves for some ranges of parameters.
\item Stability (including transverse stability ) issues.
\end{itemize}

\section{Appendix A}
In this Appendix we  state and prove the global existence of small solutions  for the one-dimensional Hamiltonian ($b=d$) system \eqref{BF} for some specific choices of the coefficients, inspired by a similar result for the classical Boussinesq systems (see \cite{BCS2} section 4.2). 

We thus consider the one-dimensional version of \eqref{BF} where we skip the underscore $\beta$, that is

\begin{equation}\label{BF1d}
\left\{\begin{array}{lll}
(1-\mu b\partial_x^2)\partial_t\zeta + \frac1\gamma ((1-\epsilon\zeta){v})_x\\
\;\;\;\;- \frac{\sqrt{\mu}}{\gamma^2} |D| \coth(\sqrt{\mu_2}|D|)v_x + \frac{\mu}{\gamma}\big(a-\frac{1}{\gamma^2} \coth^2(\sqrt{\mu_2}|D|)v_{xxx}=0,\\
(1-\mu d\partial_x)\partial_t {\bf{v}}_\beta+ (1-\gamma)  \zeta_x - \frac{\epsilon}{\gamma}vv_x + \mu c(1-\gamma) \zeta_{xxx}=0,
\end{array}\right.
\end{equation}

leading to the Hamiltonian

\begin{multline}
 H(\zeta, v)=\int_{\mathbb R} (\frac{1-\gamma}{2}\zeta^2+\frac{1}{2}| v|^2-\frac{\epsilon}{2\gamma}\zeta v^2-\frac{\mu c}{2} (1-\gamma)| \zeta_x|^2\\-\frac{a\mu}{2\gamma}| v_x|^2+\frac{\sqrt \mu}{2\gamma^2}|L_1^{1/2}v|^2+\frac{\mu}{2\gamma^3}|L_2^{1/2}  v|^2), 
 \end{multline}

Throughout this Appendix we will assume that 

$$b=d>0, \;a\leq 0,\; c< 0.$$

As in \cite{BCS2} we start with the non degenerate case 
\begin{equation}\label{nondeg}
b=d>0,\;a<0,\;c<0.
\end{equation}

\begin{teo}\label{global}
Assume that \eqref{nondeg} holds and suppose that $(\zeta_0,v_0)\in H^s(\R)\times H^s(\R),\; s\geq 1$ is such that 
\begin{equation}\label{condH}
|H(\zeta_0,v_0)|<\frac{\gamma^2(1-\gamma)\sqrt{\mu|c|}}{\epsilon^2}, \quad \inf_{x\in \R}(1-\frac{\epsilon}{\gamma}\zeta_0(x))>0.
\end{equation}
Then the corresponding solution $(\zeta,v)$ of the one-dimensional system is global in $H^s(\R)\times H^s(\R)$ and is furthermore uniformly bounded in $H^1(\R)\times H^1(\R).$

\end{teo}
\begin{proof}
We follow closely the argument used in \cite{BCS2} for some Hamiltonian Boussinesq systems. In order to avoid technicalities we will focus on the case $s=1.$

Using \eqref{condH} and the the fact that the local solution is continuous in time with value in $H^1(\R)$ we infer that there exists $t_0>0$ such that the local solution satisfies $1-\frac{\epsilon}{2\gamma}\zeta(x,t)>0$ for all $x\in \R$ and  $0\leq t<t_0$  so that for the same values of $(x,t)$ we have

%\begin{align}
\begin{equation}
\begin{array}{lll}
\zeta^2(x,t)&\leq  \int_{\mathbb R} |\zeta\zeta_x|dx=\frac{1}{\sqrt{\mu|c|}}\int_{\mathbb R} \sqrt{\mu|c|}|\zeta \zeta_x| dx\leq \frac{1}{2\sqrt{\mu |c|}}\int_{\mathbb R}(\zeta ^2+\mu|c|\zeta_x^2)dx\\
&\leq \frac{1}{(1-\gamma)\sqrt{\mu|c|}}H((\zeta(\cdot,t),v(\cdot,t))\leq \frac{1}{(1-\gamma)\sqrt{\mu|c|}}H(\zeta_0,v_0)\equiv \alpha^2
\end{array}
\end{equation}

Assuming \eqref{condH}, it follows that $\alpha^2 <\frac{\gamma^2}{\epsilon^2},$  implying that

\begin{equation}\label {condeta}
\sup_{x\in \R}|\zeta (x,t)|\leq \alpha<\frac{\gamma}{\epsilon},\quad \text {for}\; 0\leq t\leq t_0.
\end{equation}

Therefore, one has $1-\frac{\epsilon}{\gamma}\zeta(x,t)>0$ for all $x\in \R.$ Moreover, as long as the solution continues to exist in $H^1(\R)$, this positive lower bound on $1-\frac{\epsilon}{\zeta}$ continues to hold, independentely of $t\geq 0,$ by reapplying the above argument.

The uniform $H^1$ bound results then for the conservation of the Hamiltonian. Once the $H^1(\R)$ norm of $\zeta$ and $v$ are known to be uniformly bounded   a standard argument (see the proof of Theorem 4.2 in \cite{BCS2}) implies that the $H^s(\R, s>1$ of  $\zeta$ and $v$ remain bounded on bounded time intervals provided $(\zeta_0,v_0)\in H^s(\R^2)\times H^s(\R).$

\end{proof}

We now turn  to the degenerate case 
\begin{equation}\label{deg}
b=d>0,\; a=0,\;c<0.
\end{equation}
The result is now
\begin{teo}
Assume that \eqref{deg} holds and suppose that $(\zeta_0,v_0)\in H^{s+1}(\R)\times H^{s+1/2}(\R),\; s\geq 1$ is such that \eqref{condH} holds.

%\begin{equation}\label{condH}
%|H(\zeta_0,v_0)|<, \quad \inf_{x\in \R}(1+\zeta_0(x))>0.
%\end{equation}
Then the corresponding solution $(\zeta,v)$ of the one-dimensional system is global in $H^{s+1}(\R)\times H^{s+1/2}(\R)$ and is furthermore uniformly bounded in $H^1(\R)\times H^{1/2}(\R).$

\end{teo}

\begin{proof}
The proof is similar to that of Theorem 9.1 (see also Theorem 4.3 in \cite{BCS2}).
\end{proof}

\begin{remark}
It has been recently proven (\cite{KMPP}) that for some of the Hamiltonian Boussinesq abcd systems having global small solutions (see \cite{BCS2}), those solutions decay to zero, locally strongly in the energy space, uniformly in proper subsets of the light cone $|x|\leq |t].$ It would be interesting to extend those scattering results  to the systems studied in this Appendix.
\end{remark}

\section{Appendix B}

In this Appendix, we state some facts from perturbation theory of
closed linear operator on Hilbert spaces that we have used along this work
(see Kato \cite{K} for details).

We consider $L^2(\mathbb R)\times L^2(\mathbb R)$ the Hilbert space with norm defined by
$\|( f, g)\|^2=\|f\|^2+\|g\|^2$, and for any closed operator $T$ on $L^2(\mathbb R)$ with
domain $D(T)$, its graph, $\mathcal G(T)=\{(f, g)\in L^2(\mathbb R)\times L^2(\mathbb R) : f\in D(T), T(f)=g\}$. Then a metric $\widehat{\delta}$ on $\mathcal C\equiv\mathcal C(L^2(\mathbb R))$, the space of closed operator on $L^2(\mathbb R)$, may be defined as follows: for any $S, T \in \mathcal C$,
$$
\widehat{\delta}(S,T)=\|P_S - P_T\|_{B(L^2\times L^2)},
$$
where $P_S$ and $P_T$ are the orthogonal projections on $\mathcal G(S)$ and $\mathcal G(T)$, respectively, and

 $\|\cdot\|_{B(L^2\times L^2)}$ denotes the operator norm on the space of bounded operators on $L^2(\mathbb R)\times L^2(\mathbb R)$.  The results of section IV-4 of Kato \cite{K} imply the following.

\begin{teo}\label{A}
Let $S, T \in \mathcal C$, and suppose $A$ is a bounded operator
on $L^2(\mathbb R)$ with operator norm $\|A\|_{B(L^2)}$. Then
\begin{enumerate}
\item[(a)] $\widehat{\delta}(T+A,T)\leqq \|A\|_{B(L^2)}$.
%\item[(b)] $\widehat{\delta}(S+A, T+A)\leqq 2(1+\|A\|^2_{B(L^2)}) \widehat{\delta}(S,T)$.
\item[(b)] $\widehat{\delta}(S+A,T+A)\leqq2(1+||A||^2_{B(L^2})\widehat{\delta}(S,T).$
\end{enumerate}
\end{teo}

\vskip0.2in

\indent\textbf{Acknowledgements:}  This work was initiated  while the first author was visiting the Department of Mathematics of Paris-Sud University as a visiting professor and he was partially supported by CNPq/Brazil  and FAPESP (S\~ao Paulo Research Fundation/Brazil) under the process 2016/07311-0. The second author acknowledges partial support from the ANR project GEODISP and the program MathAmSud EEQUADD and the hospitality of the Wolfgang Pauli Institute, Vienna.

%\begin{thebibliography}
%{BLS}
\bibliographystyle{amsplain}

\begin{thebibliography}{99}
\bibitem{AB}\textsc{ J. P. Albert and J. L. Bona},{\it  Total Positivity and the Stability of Internal Waves in Stratified
Fluids of Finite Depth}, IMA  J. Appl. Math. {\bf 46} (1-2), (1991), 1-19.
\bibitem{ABH}\textsc{J.P. Albert, J.L. Bona and D. Henry}, {\it  Sufficient conditions for stability of solitary-wave
solutions of model equations for long waves}, Physica D {\bf 24} (1987), 343-366.
 \bibitem{ABS}\textsc{J. Albert, J.L. Bona and J.-C. Saut}, {\it Model equations for
waves in stratified fluids}, 
Proc. Royal Soc. London A, {\bf 453}, (1997), 1233--1260. 
\bibitem{AT}\textsc{ J.P. Albert and J.F. Toland},{\it  On the exact solutions of the intermediate long-wave equation.
Diff. Int. Eq.}, {\bf 7} (3-4) (1994), 601-612.
\bibitem{AT2}\textsc{ J.P. Albert and J.F. Toland},{\it Uniqueness and related analytic properties for the Benjamin-Ono
equation-a nonlinear Neumann problem in the plane}, Acta Math. {\bf 167} (1991), 107-126.
\bibitem{An} {\sc J. Angulo},  \textit{Nonlinear Dispersive Equations:
Existence and Stability of Solitary and Periodic Travelling Wave Solutions},
Mathematical Surveys and Monographs (SURV), 156,  AMS, (2009).
\bibitem{B} \textsc{T.B. Benjamin},
{\it Internal waves of permanent form in fluids of great depth}, J.
Fluid Mech. {\bf 29} (1967) 559-592.

\bibitem{BCS1} \textsc {J. L. Bona, M. Chen and J.-C. Saut}, {\it Boussinesq equations and
other systems for small-amplitude long waves in nonlinear dispersive
media I : Derivation and the linear theory}, J. Nonlinear Sci. {\bf
12} (2002), 283-318.
\bibitem{BCS2} \textsc{J. L. Bona, M. Chen and J.-C. Saut},
{\it Boussinesq equations and other systems for small-amplitude long
waves in nonlinear dispersive media. II : The nonlinear theory},
Nonlinearity {\bf 17} (2004) 925-952.
\bibitem{BCL} \textsc {J. L.  Bona, T. Colin and D. Lannes}, {\it Long-wave approximation for water waves},
Arch. Ration. Mech. Anal. {\bf 178}, (2005), 373-410.
\bibitem{BLS} {\sc J.L. Bona, D. Lannes and J.-C. Saut},  \textit{Asymptotic models  for internal waves}, J. Math. Pures Appl., 89, (2008), 538-566.

\bibitem{BL} {\sc J.L. Bona and Y. A. Li},  \textit{Decay and Analyticity of solitary waves}, J. Math. Pures Appl., 76, (1997), 377-430.
\bibitem{BS}\textsc{J.L. Bona and A.  Soyeur}, {\it  On the stability of solitary-wave solutions of model equations for
long-waves},  J. Nonlinear Sci. {\bf 4} (1994), 449-470.
\bibitem{BSS}\textsc{J. L. Bona, P.E. Souganidis and W.A. Strauss},{\it  Stability and instability of solitary waves of
KdV type}, Proc. Roy. Soc. London A {\bf 411}, (1987)395-412.
\bibitem{deBS} \textsc{A. de Bouard and  J.-C. Saut}, {\it Symmetry and decay of the generalized Kadomtsev-Petviashvili solitary waves} SIAM J. Math. Anal. {\bf 28}, 5 (1997), 104-1085
\bibitem{Bu}\textsc{C. Burtea}, {\it New long time existence results for a class of Boussinesq-type systems}, J. Math. Pures Appl. {\bf 106} (2) (2016), 203-236.



\bibitem{C} {\sc A.P. Calderon},  \textit{Commutators of singular integrals operators}, Proc. Nat. Acad. Sci. U.S.A., 53, (1965), 1092-1099.
\bibitem{CC1} \textsc{W. Choi and R. Camassa}, {\it Long internal waves of finite amplitude}, Phys. Rev. Letters {\bf 77} (9) (1996), 1759-1762.
\bibitem{ChCa}\sc{W. Choi and R. Camassa}, {\it Weakly nonlinear internal waves in a two-fluid system}, J. Fluid. Mech. {\bf 313} (1996), 83-103.
\bibitem{CC2} \textsc{W. Choi and R. Camassa},
{\it Fully nonlinear internal waves in a two-fluid system}, J. Fluid
Mech. {\bf 396} (1999) 1-36.
\bibitem{CM} {\sc R. Coifman and Y. Meyer},  \textit{Au del\`a des op\'erateurs pseudo-diff\'erentiels}, Asterisque, no 57, Soci\'et\'e Math\'ematique de France,  Paris (1978).
\bibitem{CGK} \textsc{W. Craig, P. Guyenne, and H. Kalisch},
{\it Hamiltonian long-wave expansions for free surfaces and
interfaces}, Comm. Pure. Appl. Math. {\bf 58} (2005)1587-1641.

\bibitem{Cung} {\sc Cung the Anh},  \textit{On the Boussinesq-full dispersion system and Boussinesq-Boussinesq systems for internal waves}, Nonlinear Analysis, 72, 1 (2010), 409-429.
\bibitem{Cung2} {\sc Cung the Anh},  \textit{Influence of surface tension and bottom topography on internal waves}, Mathematics Models and Methods in Applied Sciences, {\bf 10} (12)(2009), 2145-2175.
\bibitem{DM}\textsc{A. Duran and D. Mitsotakis}, {\it Solitary wave solutions of Intermediate Long Wave and Benjamin-Ono systems for internal waves}, unpublished manuscript (2015).
\bibitem{FL} {\sc R.L. Frank and E. Lenzmann},  \textit{Uniqueness of non-linear ground states for fractional Laplacians in $\mathbb R$}, Acta Math., 210, 2 (2013), 261-318.

\bibitem{K} {\sc T. Kato},  \textit{Perturbation theory for linear Operators}, Springer, Berlin, (1976).

\bibitem{KM} \textsc{C.E. Kenig and Y. Martel}, {\it  Asymptotic stability of solitons for the Benjamin-Ono equation},
Revista Matematica Iberoamericana {\bf 25} (2009) 909-970.

\bibitem{KS}\textsc{C. Klein and J.-C. Saut}, {\it IST versus PDE, a comparative study}, in  {\it Hamiltonian Partial Differential Equations and Applications}, Fields Institute Communications, vol. 75 (2015), 383-449.

%\bibitem{Li} {\sc Li Xu},  \textit{Intermediate long waves systems for internal waves}, Nonlinearity, 25 (2012), 597-640.

\bibitem{KKD} \textsc{T. Kubota,  D. R. S. Ko and L. Dobbs},
{\it Propagation of weakly nonlinear internal waves in a stratified
fluid of finite depth}, J. Hydronautics {\bf 12} (1978) 157-165.
\bibitem{KMPP}\textsc{C. Kwak, C. Munoz, F. Poblete and J.C. Pozo}, {\it The scattering problem for the abcd Boussinesq system in the energy space}, arXiv:1712.09256v2,  28 Dec 2017.[math.AP].
\bibitem{Lions} {\sc P.-L. Lions},  \textit{The concentration-compactness principle in the calculus of variations. The locally compactness case, part 1}, Ann. Inst. H. Poincar\'e, Anal. Non Lin\'eare, 1 (1984), 109-145.


\bibitem{Sa}\textsc{J.-C. Saut}, {\it Lectures on Asymptotic Models for Internal Waves}, in Lectures on the Analysis of Nonlinear Partial Differential Equations Vol.2 MLM2, Higher Education Press and International Press, Beijing-Boston (2011), 147-201.
\bibitem{SWX}\textsc{J.-C. Saut, Chao Wang and Li Xu}, {\it  The Cauchy problem on large time for surface waves Boussinesq systems II}, SIAM J. Math. Anal. {\bf 49} (4) (2017), 2321-2386.
\bibitem{SX}\textsc{J.-C. Saut  and Li Xu}, {\it The Cauchy problem on large time for surface waves Boussinesq systems}, J. Math.Pures et Appl. {\bf 97} (2012), 635-662.

\bibitem{Xu} {\sc Li Xu},  \textit{Intermediate long waves systems for internal waves}, Nonlinearity, 25 (2012), 597-640.
\end{thebibliography}

\end{document}